\documentclass[11pt,reqno]{amsart}

\newtheorem{proposition}{Proposition}[section]
\newtheorem{lemma}[proposition]{Lemma}
\newtheorem{corollary}[proposition]{Corollary}
\newtheorem{theorem}[proposition]{Theorem}

\theoremstyle{definition}
\newtheorem{definition}[proposition]{Definition}
\newtheorem{example}[proposition]{Example}
\newtheorem{examples}[proposition]{Examples}

\newcommand{\thlabel}[1]{\label{th:#1}}
\newcommand{\thref}[1]{Theorem~\ref{th:#1}}
\newcommand{\selabel}[1]{\label{se:#1}}
\newcommand{\seref}[1]{Section~\ref{se:#1}}
\newcommand{\lelabel}[1]{\label{le:#1}}
\newcommand{\leref}[1]{Lemma~\ref{le:#1}}
\newcommand{\prlabel}[1]{\label{pr:#1}}
\newcommand{\prref}[1]{Proposition~\ref{pr:#1}}
\newcommand{\colabel}[1]{\label{co:#1}}
\newcommand{\coref}[1]{Corollary~\ref{co:#1}}

\newcommand{\exlabel}[1]{\label{ex:#1}}
\newcommand{\exref}[1]{Example~\ref{ex:#1}}
\newcommand{\delabel}[1]{\label{de:#1}}
\newcommand{\deref}[1]{Definition~\ref{de:#1}}
\newcommand{\eqlabel}[1]{\label{eq:#1}}
\newcommand{\equref}[1]{(\ref{eq:#1})}

\def\ot{\otimes}

\def\CC{{\mathbb C}}

\newcommand{\Cc}{\mathcal{C}}

\newcommand{\Mm}{\mathcal{M}}

\def\*C{{}^*\hspace*{-1pt}{\Cc}}
\def\text#1{{\rm {\rm #1}}}

\input xy
\xyoption {all} \CompileMatrices

\usepackage{amssymb}
\usepackage{color,amssymb,graphicx,amscd,amsmath}
\usepackage[colorlinks,urlcolor=blue,linkcolor=blue,citecolor=blue]{hyperref}

\begin{document}
\title[Extending structures for algebras]
{Extending structures, Galois groups and supersolvable associative
algebras}

\author{A. L. Agore}
\address{Faculty of Engineering, Vrije Universiteit Brussel, Pleinlaan 2,
B-1050 Brussels, Belgium \textbf{and} Department of Applied Mathematics,
Bucharest University of Economic Studies, Piata Romana 6,
RO-010374 Bucharest 1, Romania} \email{ana.agore@vub.ac.be and
ana.agore@gmail.com}

\author{G. Militaru}
\address{Faculty of Mathematics and Computer Science, University of Bucharest, Str.
Academiei 14, RO-010014 Bucharest 1, Romania}
\email{gigel.militaru@fmi.unibuc.ro and gigel.militaru@gmail.com}

\thanks{A.L. Agore is Postoctoral Fellow of the Fund for Scientific
Research-Flanders (Belgium) (F.W.O. Vlaanderen). This work was
supported by a grant of the Romanian National Authority for
Scientific Research, CNCS-UEFISCDI, grant no. 88/05.10.2011.}

\subjclass[2010]{16D70, 16Z05, 16E40} \keywords{unified products,
Galois groups, classification of algebras}


\maketitle

\begin{abstract}
Let $A$ be a unital associative algebra over a field $k$. All
unital associative algebras containing $A$ as a subalgebra of a
given codimension $\mathfrak{c}$ are described and classified. For
a fixed vector space $V$ of dimension $\mathfrak{c}$, two
non-abelian cohomological type objects are explicitly constructed:
${\mathcal A}{\mathcal H}^{2}_{A} \, (V, \, A)$ will classify all
such algebras up to an isomorphism that stabilizes $A$ while
${\mathcal A}{\mathcal H}^{2} \, (V, \, A)$ provides the
classification from H\"{o}lder's extension problem viewpoint. A
new product, called the unified product, is introduced as a tool
of our approach. The classical crossed product or the twisted
tensor product of algebras are special cases of the unified
product. Two main applications are given: the Galois group ${\rm
Gal} \, (B/A)$ of an extension $A \subseteq B$ of associative
algebras is explicitly described as a subgroup of a semidirect
product of groups ${\rm GL}_k (V) \rtimes {\rm Hom}_k (V, \, A)$,
where the vector space $V$ is a complement of $A$ in $B$. The
second application refers to supersolvable algebras introduced as
the associative algebra counterpart of supersolvable Lie algebras.
Several explicit examples are given for supersolvable algebras
over an arbitrary base field, including those of characteristic
two whose difficulty is illustrated.
\end{abstract}

\section*{Introduction}
The H\"{o}lder's \emph{extension problem} at the level of unital
associative algebras was initiated by Everett \cite{Ev} and
Hochschild \cite{Hoch2} and is still an open and notoriously
difficult problem. Rephrased in its full generality, it consist of
the following question: \emph{let $A$ be an algebra over a field
$k$, $E$ a vector space and $\pi : E \to A$ a $k$-linear
epimorphism of vector spaces. Describe and classify all algebra
structures that can be defined on $E$ such that $\pi : E \to A$ is
a morphism of algebras.} The partial answer to the extension
problem was given in \cite[Theorem 6.2]{Hoch2}: all algebra
structures $\cdot$ on $E$ such that $W := {\rm Ker} (\pi)$ is a
two-sided ideal of null square are classified by the second
Hochschild cohomological group ${\rm H}^2 (A, W)$. The elementary
way in which we reformulated the extension problem allows us to
consider the categorical dual of it. Thus, we arrive at the
following question which is the subject of the present paper:

\textbf{Extending structures problem.} \textit{Let $A$ be a unital
associative algebra and $E$ a vector space containing $A$ as a
subspace. Describe and classify the set of all unital associative
algebra structures $\cdot$ that can be defined on $E$ such that
$A$ becomes a subalgebra of $(E, \cdot)$.}

The extending structures (ES) problem is the associative algebra
counterpart of the similar problem studied for Lie algebras in
\cite{am-2013} but, as we will see later on, the associative
algebra version of it is far from being trivial since the
anti-symmetry on the bracket simplifies the construction performed
for Lie algebras. If we fix $V$ a complement of $A$ in $E$ then
the ES-problem can be rephrased equivalently as follows: describe
and classify all unital associative algebras containing $A$ as a
subalgebra of codimension equal to ${\rm dim} (V)$. The ES-problem
is the dual of the extension problem and it also generalizes the
so-called \emph{radical embedding problem} introduced by Hall
\cite{hall} which consists of the following: given a finite
dimensional nilpotent algebra $A$, describe the family of all
unital algebras with radical isomorphic to $A$ - for more details
see \cite{fl}. The ES-problem is a very difficult question: if $A
= k$ then the ES-problem asks in fact for the classification of
all algebra structures on a given vector space $E$. For this
reason, from now on we will assume that $A \neq k$. Regardless of
the difficulty of the ES-problem, we can still provide detailed
answers to it in certain special cases which depend on the choice
of the algebra $A$ and mainly on the codimension of $A$ in $E$.
Moreover, a new class of algebras, which we call
\emph{supersolvable} algebras, appear on the route as the
associative algebra counterpart of supersolvable Lie algebras
\cite{barnes}: these are finite dimensional algebras $E$ for which
there exists a finite chain of subalgebras $ E_0 :=k \subset E_1
\subset \cdots \subset E_m := E$ such that $E_i$ has codimension
$1$ in $E_{i+1}$, for all $i = 0, \cdots, m-1$. All supersolvable
algebras of a given dimension can be classified using the method
introduced in this paper by a recursive type algorithm in which
the key step is the one for codimension $1$: the crucial role will
be played by the characters of the algebra $A$ and by some twisted
derivations of $A$ satisfying certain compatibility conditions.

The paper is organized as follows. In \seref{unifiedprod} we will
perform the abstract construction of the unified product $A
\ltimes V$: it is associated to an algebra $A$, a vector space $V$
and a system of data $\Omega(A, V) = \bigl(\triangleleft, \,
\triangleright, \, \leftharpoonup, \, \rightharpoonup, \, f, \,
\cdot \bigl)$ called an extending datum of $A$ through $V$.
\thref{1} establishes the set of axioms that need to be satisfied
by $\Omega(A, V)$ such that $A\ltimes V$ with a certain
multiplication becomes an associative unitary algebra, i.e. a
unified product. In this case, $\Omega(A, V) =
\bigl(\triangleleft, \, \triangleright, \, \leftharpoonup, \,
\rightharpoonup, \, f, \, \cdot \bigl)$ will be called an
\emph{algebra extending structure} of $A$ through $V$. We
highlight here an important difference between the unified product
for algebras and the corresponding unified products for Lie
algebras \cite{am-2013}. The construction of the unified product
$A\ltimes V$ for algebras requires two more actions
$\leftharpoonup \, : A \times V \to A$ and $\rightharpoonup \, : A
\times V \to V$ which connects $A$ and $V$. The two actions are
missing in the case of Lie algebras: the absence of the two maps
is a consequence of the anti-symmetry on the bracket. Now let $A$
be an algebra, $E$ a vector space containing $A$ as a subspace and
$V$ a given complement of $A$ in $E$. \thref{classif} provides the
answer to the description part of the ES-problem: there exists an
algebra structure $\cdot$ on $E$ such that $A$ is a subalgebra of
$(E, \cdot)$ if and only if there exists an isomorphism of
algebras $(E, \cdot) \cong A \ltimes V$, for some algebra
extending structure $\Omega(\mathfrak{g}, V) =
\bigl(\triangleleft, \, \triangleright, \, \leftharpoonup, \,
\rightharpoonup, \, f, \, \cdot \bigl)$ of $A$ through $V$.
Moreover, the algebra isomorphism $(E, \cdot) \cong A \ltimes V$
can be chosen in such a way that it stabilizes $A$ and
co-stabilizes $V$. Based on this result we are able to give the
theoretical answer to the ES-problem in \thref{main1}: all algebra
structures on $E$ containing $A$ as a subalgebra are classified by
two non-abelian cohomological type objects which are explicitly
constructed. The first one is denoted by ${\mathcal A}{\mathcal
H}^{2}_{A} \, (V, \, A)$ and will classify all such structures up
to an isomorphism that stabilizes $A$. We also indicate the
bijection between the elements of ${\mathcal A}{\mathcal
H}^{2}_{A} \, (V, \, A)$ and the isomorphism classes of all
extending structures of $A$ to $E$. Having in mind that we want to
\emph{extend} the algebra structure on $A$ to a bigger vector
space $E$ this is in fact the object responsible for the
classification of the ES-problem. If $V = k^n$, the object
${\mathcal A}{\mathcal H}^{2}_{A} \, (k^n, \, A)$ classifies up to
an isomorphism all algebras which contain and stabilize $A$ as a
subalgebra of codimension $n$. Hence, by computing ${\mathcal
A}{\mathcal H}^{2}_{A} \, (k^n, \, A)$, for a given algebra $A$,
we obtain important information regarding the classification of
finite dimensional algebras. On the other hand, in order to comply
with the traditional way of approaching the extension problem
\cite{Hoch2}, we also construct a second classifying object,
denoted by ${\mathcal A}{\mathcal H}^{2} \, (V, \, A)$, which
provides the classification of the ES-problem up to an isomorphism
of algebras that stabilizes $A$ and co-stabilize $V$. Thus, the
object ${\mathcal A}{\mathcal H}^{2} \, (V, \, A)$, whose
construction is simpler than the one of ${\mathcal A}{\mathcal
H}^{2}_{A} \, (V, \, A)$, appears as a sort of dual of the
classical Hochschild cohomological group. There exists a canonical
projection $ {\mathcal A}{\mathcal H}^{2} \, (V, \, A)
\twoheadrightarrow {\mathcal A}{\mathcal H}^{2}_A \, (V, \, A)$
between these two classifying objects. Although the results
presented in \seref{unifiedprod} are of a rather technical nature,
their efficiency and applicability will be shown in
\seref{cazurispeciale} and \seref{exemple} where several examples
are provided. As a first surprising application of the classifying
object ${\mathcal A}{\mathcal H}^{2}_{A} \, (V, \, A)$, we are
able to describe the Galois group ${\rm Gal} \, (B/A)$ of an
arbitrary extension $A \subseteq B$ of associative algebras as a
subgroup of a semidirect product of groups ${\rm GL}_k (V) \rtimes
{\rm Hom}_k (V, \, A)$, where the vector space $V$ is isomorphic
to the quotient vector space $B/A$ (\coref{galgrup}). In
particular, for any extension of finite dimensional algebras
$A\subseteq B$, the Galois group ${\rm Gal} \, (B/A)$ embeds in
${\rm GL} (m, \, k) \rtimes {\rm M}_{n\times m} (k)$, for some
positive integers $m$ and $n$, where ${\rm M}_{n\times m} (k)$ is
the additive group of $(n\times m)$-matrices over $k$. This shows
that, for a potential generalization of the classical Galois
theory at the level of associative algebras the role of the
symmetric groups $S_n$ has to be replaced by these canonical
semidirect products of groups ${\rm GL} (m, \, k) \rtimes {\rm
M}_{n\times m} (k)$. The next part of the section deals with
several special cases of the unified product and we emphasize the
problem for which each of these products is responsible. Let $i: A
\hookrightarrow E$ be an inclusion of algebras. \coref{splitleft},
\coref{splitbimo} and \coref{splialg} give necessary and
sufficient conditions for $i$ to have a retraction which is a
left/right $A$-linear map, an $A$-bimodule map and respectively an
algebra map. In the latter case, the associated unified product
has a simple form which we call the \emph{semidirect product} by
analogy with the group and Lie algebra case since it describes the
split monomorphisms in the category of algebras. We also show that
the classical crossed products \cite{passman} and their
generalizations \cite{TB, cap} as well as the Ore extensions are
all special cases of the unified product. \deref{mpalgebras}
introduces a new concept, namely the \emph{matched pair} of two
algebras $A$ and $V$. As a special case of the unified product, we
define the \emph{bicrossed product} $A \bowtie V$ associated to a
matched pair of algebras. Even if the definition differs from the
one given in the case of groups \cite{Takeuchi} or Lie algebras
\cite{majid}, \coref{bicrfactor} shows that it plays the same
role, namely it provides the tool to answer the factorization
problem for algebras. The end of the section deals with the
commutative case: in this case the unified product and its axioms
simplify considerably.

Computing the classifying objects ${\mathcal A} {\mathcal
H}^{2}_{A} \, (V, \, A)$ and ${\mathcal A} {\mathcal H}^{2} \, (V,
\, A)$ is a highly nontrivial problem. In \seref{exemple} we shall
identify a way of computing both objects for what we have called
\emph{flag extending structures} of $A$ to $E$ as defined in
\deref{flagex}. All flag extending structures of $A$ to $E$ can be
completely described by a recursive reasoning where the key step
is the case when $A$ has codimension $1$ as a subspace of $E$.
This case is solved in \thref{clasdim1}, where ${\mathcal A}
{\mathcal H}^{2}_{A} \, (k, \, A)$ and ${\mathcal A} {\mathcal
H}^{2} \, (k, \, A)$ are described. The key players in this
context are the so-called flag datums of an algebra $A$ introduced
in \deref{tehnica}: in the definition of a flag datum two
characters of the algebra $A$ are involved as well as two twisted
derivations satisfying certain compatibility conditions. As an
application, in \coref{galcodim1} we shall compute the Galois
group for any algebra extension $A \subseteq B$ for which $A$ has
codimension $1$ in $B$. \thref{clasdim1} proves itself to be
efficient for the classification of all supersolvable algebras.
\coref{dim2alg} classifies and counts the number of types of
isomorphisms of algebras of dimension $2$ over an arbitrary field
$k$. By iterating the algorithm, we can increase the dimension by
$1$ at each step, obtaining in this way the classification of all
supersolvable algebras in dimension $3$, $4$, etc. (see
\thref{flag3dimsin} for dimension $3$). We mention that all
classification results presented in this paper are over an
arbitrary field $k$, including the case of characteristic $2$
whose difficulty is illustrated.


\section{Unified products for algebras}\selabel{unifiedprod}
\subsection*{Notations and terminology}
For a family of sets $(X_i)_{i\in I}$ we shall denote by
$\sqcup_{i\in I} \, X_i$ their coproduct in the category of sets,
i.e. $\sqcup_{i\in I} \, X_i$ is the disjoint union of the
$X_i$'s. All vector spaces, linear or bilinear maps are over an
arbitrary field $k$. A map $f: V \to W$ between two vector spaces
is called the trivial map if $f (v) = 0$, for all $v\in V$. By an
\emph{algebra} $A$ we mean an unital associative algebra over $k$.
However, when the algebras are not unital it will be explicitly
mentioned. All algebra maps preserve units and any left/right
$A$-module is unital. ${\rm Alg} \,(A, k)$ denotes the space of
all algebra maps $A \to k$ and ${\rm Aut}_{\rm Alg} (A)$ the group
of algebra automorphisms of $A$. For an algebra $A$ we shall
denote by ${}_A\Mm_A$ the category of all $A$-bimodules, i.e.
triples $(V, \, \rightharpoonup, \, \triangleleft)$ consisting of
a vector space $V$ and two bilinear maps $\rightharpoonup \, : A
\times V \to V$, $\triangleleft : V \times A \to V$ such that $(V,
\rightharpoonup)$ is a left $A$-module, $(V, \triangleleft)$ is a
right $A$-module and $a \rightharpoonup (x \triangleleft b) = (a
\rightharpoonup x) \triangleleft b$, for all $a$, $b\in A$ and
$x\in V$. If $(V, \, \rightharpoonup, \, \triangleleft) \in
{}_A\Mm_A$, then the \emph{trivial extension} of $A$ by $V$ is the
algebra $A \times V$, with the multiplication given for any $a$,
$b\in A$, $x$, $y\in V$ by:
$$
(a, \, x) \cdot (b, \, y) := \bigl(ab, \,\, a \rightharpoonup y  +
x\triangleleft b \bigl)
$$
Let $A$ be an algebra, $E$ a vector space such that $A$ is a
subspace of $E$ and $V$ a complement of $A$ in $E$, i.e. $V$ is a
subspace of $E$ such that $E = A + V$ and $A \cap V = \{0\}$. For
a linear map $\varphi: E \to E$ we consider the diagram:
\begin{eqnarray} \eqlabel{diagrama1}
\xymatrix {& A \ar[r]^{i} \ar[d]_{Id} & {E}
\ar[r]^{\pi} \ar[d]^{\varphi} & V \ar[d]^{Id}\\
& A \ar[r]^{i} & {E}\ar[r]^{\pi } & V}
\end{eqnarray}
where $\pi : E \to V$ is the canonical projection of $E = A + V$
on $V$ and $i: A \to E$ is the inclusion map. We say that
$\varphi: E \to E$ \emph{stabilizes} $A$ (resp.
\emph{co-stabilizes} $V$) if the left square (resp. the right
square) of diagram \equref{diagrama1} is commutative.

Two algebra structures $\cdot $ and $\cdot'$ on $E$ containing $A$
as a subalgebra are called \emph{equivalent} and we denote this by
$(E, \cdot ) \equiv (E, \cdot')$, if there exists an algebra
isomorphism $\varphi: (E, \cdot) \to (E, \cdot')$ which stabilizes
$A$. The algebra structures $\cdot$ and $\cdot'$ on $E$ are called
\emph{cohomologous} and we denote this by $(E, \cdot) \approx (E,
\cdot')$, if there exists an algebra isomorphism $\varphi: (E,
\cdot) \to (E, \cdot')$ which stabilizes $A$ and co-stabilizes
$V$, i.e. diagram \equref{diagrama1} is commutative. Then $\equiv$
and $\approx$ are both equivalence relations on the set of all
algebra structures on $E$ containing $A$ as a subalgebra and we
denote by ${\rm Extd} \, (E, \, A)$ (resp. ${\rm Extd}' \, (E, \,
A)$) the set of all equivalence classes via $\equiv$ (resp.
$\approx$). ${\rm Extd} \, (E, \, A)$ is the classifying object
for the ES problem: by explicitly computing ${\rm Extd} \, (E, \,
A )$ we obtain a parametrization of the set of all isomorphism
classes of algebra structures on $E$ which contain and stabilize
$A$ as a subalgebra. ${\rm Extd}' \, (E, \, A)$ gives a more
restrictive classification of the ES problem, similar to the
approach used in the case of the extension problem. Any two
cohomologous algebra structures on $E$ are of course equivalent,
hence there exists a canonical projection ${\rm Extd}' \,  (E, \,
A ) \twoheadrightarrow {\rm Extd} \, (E, \, A )$. In order to give
the answer to the ES-problem we introduce the following:

\begin{definition}\delabel{exdatum}
Let $A$ be an algebra and $V$ a vector space. An \textit{extending
datum of $A$ through $V$} is a system $\Omega(A, V) =
\bigl(\triangleleft, \, \triangleright, \, \leftharpoonup, \,
\rightharpoonup, \, f, \, \cdot \bigl)$ consisting of six bilinear
maps
\begin{eqnarray*}
\triangleleft : V \times A \to V, \quad &&\triangleright : V \times
A \to A, \quad \leftharpoonup \, : A \times V \to A, \quad
\rightharpoonup \, : A \times V \to V \\
&& f: V\times V \to A, \quad \cdot \, : V\times V \to V
\end{eqnarray*}
The extension datum $\Omega(A, V)$ is called \emph{normalized} if
for any $x \in V$ we have:
\begin{equation}\eqlabel{normed}
x \triangleright 1_A = 0, \,\,\,\, x \triangleleft 1_A = x,
\,\,\,\, 1_A \leftharpoonup x = 0, \,\,\,\, 1_A \rightharpoonup x
= x
\end{equation}

Let $\Omega(A, V) = \bigl(\triangleleft, \, \triangleright, \,
\leftharpoonup, \, \rightharpoonup, \, f, \, \cdot \bigl)$ be an
extending datum.  We denote by $ A \, \ltimes_{\Omega(A, V)} V = A
\, \ltimes V$ the vector space $A \, \times V$ together with the
bilinear map $\bullet$ defined by:
\begin{equation}\eqlabel{produnif}
(a, \, x) \bullet (b, \, y) := \bigl( ab + a \leftharpoonup y + x
\triangleright b + f(x, y), \,\, a \rightharpoonup y  +
x\triangleleft b + x \cdot y \bigl)
\end{equation}
for all $a$, $b \in A$ and $x$, $y \in V$. The object $A\ltimes V$
is called the \textit{unified product} or the \emph{dual
Hochschild product} of $A$ and $\Omega(A, V)$ if it is an
associative algebra with the multiplication given by
\equref{produnif} and the unit $(1_A, 0_V)$. In this case the
extending datum $\Omega(A, V) = \bigl(\triangleleft, \,
\triangleright, \, \leftharpoonup, \, \rightharpoonup, \, f, \,
\cdot \bigl)$ is called an \textit{algebra extending structure} of
$A$ through $V$. The maps $\triangleleft$, $\triangleright$,
$\leftharpoonup$ and $\rightharpoonup$ are called the
\textit{actions} of $\Omega(A, V)$ and $f$ is called the
\textit{cocycle} of $\Omega(A, V)$.
\end{definition}

The multiplication given by \equref{produnif} has a rather
complicated formula; however, for some specific elements we obtain
easier forms which will be useful for future computations:
\begin{eqnarray}
(a, 0) \bullet (b, y) &=& (ab + a \leftharpoonup y, \, a
\rightharpoonup y) \eqlabel{001}\\
(0, x) \bullet (b, y) &=& (x \triangleright b +
f(x, \, y), \, x \triangleleft b + x \cdot y) \eqlabel{002}\\
(a, x) \bullet (0, y) &=& (a \leftharpoonup y + f (x, \, y), \, a
\rightharpoonup y + x \cdot y) \eqlabel{003}\\
(a, x) \bullet (b, 0) &=& (ab + x \triangleright b, \, x
\triangleleft b)
\end{eqnarray}
for all $a$, $b \in A$ and $x$, $y \in V$. In particular, for any
$a$, $b \in A$ and $x$, $y \in V$ we have:
\begin{eqnarray}
(a, 0) \bullet (b, 0) &=& (ab , \, 0), \,\,\,\,\,\,\,\,\,\,\, (0,
x) \bullet (0,
y) = (f(x, \, y), \, x \cdot y) \eqlabel{001a}\\
(a, 0) \bullet (0, x) &=& (a \leftharpoonup x, \, a
\rightharpoonup x), \,\,\,\, (0, x) \bullet (b, 0) =
(x\triangleright b, \, x \triangleleft b) \eqlabel{002b}
\end{eqnarray}
The next theorem provides the necessary and sufficient conditions
that need to be fulfilled by an extending datum $\Omega(A, V)$
such that $A \ltimes V$ is a unified product.

\begin{theorem}\thlabel{1}
Let $A$ be an algebra, $V$ a vector space and $\Omega(A, V) =
\bigl(\triangleleft, \, \triangleright, \, \leftharpoonup, \,
\rightharpoonup, \, f, \, \cdot \bigl)$ an extending datum of $A$
by $V$. The following statements are equivalent:

$(1)$ $A \ltimes V$ is a unified product;

$(2)$ The following compatibilities hold for any $a$, $b \in A$,
$x$, $y$, $z \in V$:
\begin{enumerate}
\item[(A1)] $\Omega(A, V)$ is a normalized extending datum and
$(V, \rightharpoonup, \triangleleft) \in {}_A\Mm_A$ is an
$A$-bimodule;

\item[(A2)] $x \cdot ( y \cdot z) - (x \cdot y) \cdot z = f (x, \,
y) \rightharpoonup z \, - \,  x \triangleleft f (y, \, z)$;

\item[(A3)] $f(x, \, y\cdot z) - f(x\cdot y, \, z) = f (x, \, y)
\leftharpoonup z - x \triangleright f(y, \, z)$;

\item[(A4)] $a \rightharpoonup (x \cdot y) = (a \rightharpoonup x)
\cdot y + (a \leftharpoonup x) \rightharpoonup y$;

\item[(A5)] $(a \leftharpoonup x) \leftharpoonup y = a
\leftharpoonup (x \cdot y) + a f(x, \, y) - f (a \rightharpoonup
x, \,  y)$;

\item[(A6)] $ (ab) \leftharpoonup x = a (b \leftharpoonup x) + a
\leftharpoonup (b \rightharpoonup x)$;

\item[(A7)] $ x \triangleright (ab) = (x \triangleright a) b + (x
\triangleleft a) \triangleright b$;

\item[(A8)] $ x \triangleright (y \triangleright a) = (x \cdot y)
\triangleright a + f (x, \, y) a - f(x, \, y \triangleleft a) $;

\item[(A9)] $(x\cdot y) \triangleleft a = x \triangleleft (y
\triangleright  a) + x \cdot (y \triangleleft a)$;

\item[(A10)] $a (x \triangleright b) + a \leftharpoonup (x
\triangleleft b) = (a \leftharpoonup x) b + (a\rightharpoonup x)
\triangleright b$;

\item[(A11)] $x \triangleright (a \leftharpoonup y) + f (x, \, a
\rightharpoonup y) = (x \triangleright a) \leftharpoonup y + f (x
\triangleleft a, \, y)$;

\item[(A12)] $x \triangleleft (a \leftharpoonup y) + x \cdot (a
\rightharpoonup y) = (x \triangleright a) \rightharpoonup y + (x
\triangleleft a) \cdot y $;

\end{enumerate}
\end{theorem}

Before going into the proof of the theorem, we have a few
observations on the compatibilities in \thref{1}. Although they
look rather complicated at first sight, they are in fact quite
natural and can be interpreted as follows: (A2) measures how far
$(V, \cdot)$ is from being an associative algebra and is called
the \emph{twisted associativity condition}. The compatibility
condition (A3) is a $2$-cocycle type condition. (A5) and (A8) are
deformations of the usual module conditions and they can be called
\emph{twisted module conditions} for the actions $\leftharpoonup$
and $\triangleright$. (A4), (A6), (A7), (A9), (A10) and (A12) are
compatibility conditions between the actions $(\triangleleft, \,
\triangleright, \, \leftharpoonup, \, \rightharpoonup)$ of
$\Omega(A, V)$. They will be used in the next section as a
definition for the notion of \emph{matched pair} of algebras.

\begin{proof}
We can easily check that $(1_A, 0_V)$ is a unit for the
multiplication given by \equref{produnif} if and only if the
extending datum $\Omega(A, V)$ is normalized. The rest of the
proof relies on a detailed analysis of the associativity condition
for the multiplication given by \equref{produnif}:
\begin{equation}\eqlabel{005}
\bigl[(a,\, x) \bullet (b,\, y)\bigl] \bullet (c,\, z) = (a,\, x)
\bullet \bigl[(b,\, y) \bullet (c,\, z)\bigl]
\end{equation}
where $a$, $b$, $c \in A$ and $x$, $y$, $z \in V$. Furthermore,
since in $A \ltimes V$ we have $(a, x) = (a, 0) + (0, x)$ it
follows that \equref{005} holds if and only if it holds for all
generators of $A \ltimes V$, i.e. for the set $\{(a, \, 0) ~|~ a
\in A\} \cup \{(0, \, x) ~|~ x \in V\}$. However, since the
computations are rather long but straightforward we will only
indicate the main steps of the proof. We will start by proving
that (A2) and (A3) hold if and only if \equref{005} holds for the
triple $(0, x)$, $(0, y)$, $(0, z)$ with $x$, $y$, $z \in V$.
Indeed, we have:
\begin{eqnarray*}
\bigl[(0,\, x) \bullet (0,\, y)\bigl] \bullet (0,\, z) &=&
\bigl(f(x,\,y) \leftharpoonup z + f(x \cdot y, \, z), \, f(x, \,y)
\rightharpoonup z + (x \cdot y) \cdot z \bigl)\\
(0,\, x) \bullet \bigl[(0,\, y) \bullet (0,\, z)\bigl] &=& \bigl(x
\rhd f(y, \, z) + f(x, \, y \cdot z), \, x \lhd f(y, \, z) + x
\cdot (y \cdot z)\bigl)
\end{eqnarray*}
Therefore, (A2) and (A3) hold if and only if \equref{005} holds
for the triple $(0, x)$, $(0, y)$, $(0, z)$ with $x$, $y$, $z \in
V$. In the same manner we can prove the following: (A4) and (A5)
hold if and only if \equref{005} holds for the triple $(a, 0)$,
$(0, x)$, $(0, y)$ with $a \in A$, $x$, $y \in V$. Furthermore,
\equref{005} holds for the triple $(a, 0)$, $(b, 0)$, $(0, x)$ if
and only if (A6) holds and $\rightharpoonup$ is a left $A$-module
structure on $V$. (A7) holds, together with the fact that $\lhd$
is a right $A$-module structure on $V$, if and only if
\equref{005} holds for the triple $(0, x)$, $(a, 0)$, $(b, 0)$.
(A8) and (A9) hold if and only if \equref{005} holds for the
triple $(0, x)$, $(0, y)$, $(a, 0)$. \equref{005} holds for the
triple $(a, 0)$, $(0, x)$, $(b, 0)$ if and only if (A10) holds as
well as the compatibility condition which makes $V$ an
$A$-bimodule with respect to the actions $\rightharpoonup$ and
$\lhd$. Finally, (A11) and (A12) hold if and only if \equref{005}
holds for the triple $(0, x)$, $(a, 0)$, $(0, y)$.
\end{proof}

From now on, an algebra extending structure of $A$
through a vector space $V$ will be viewed as a system
$\Omega(A, V) = \bigl(\triangleleft, \, \triangleright, \,
\leftharpoonup, \, \rightharpoonup, \, f, \, \cdot \bigl)$
satisfying the compatibility conditions (A1)-(A12). We denote by
${\mathcal A} {\mathcal E} (A, V)$ the set of all
algebra extending structures of $A$ through $V$.

\begin{example}\exlabel{twistedproduct}
We provide the first example of an algebra extending structure and
the corresponding unified product. More examples will be given in
\seref{cazurispeciale} and \seref{exemple}.

Let $\Omega(A, V) = \bigl(\triangleleft, \, \triangleright, \,
\leftharpoonup, \, \rightharpoonup, \, f, \, \cdot \bigl)$ be an
extending datum of $A$ through $V$ such that $\triangleright$,
$\leftharpoonup$ and $\cdot$ are the trivial maps. Then,
$\Omega(A, V)$ is an algebra extending structure of $A$ through
$V$ if and only if $(V, \, \rightharpoonup, \, \triangleleft) \in
{}_A\Mm_A$ and the following compatibilities are fulfilled:
\begin{eqnarray*}
f (x, \, y) \rightharpoonup z &=&  x \triangleleft f (y, \, z),  \quad \,\,\,
a f(x, \, y) =  f (a \rightharpoonup x, \,  y) \\
f(x, \, y \triangleleft a) &=& f (x, \, y) a, \quad
f (x, \, a \rightharpoonup y)= f (x \triangleleft a, \, y)
\end{eqnarray*}
for all $a \in A$, $x$, $y$ and $z \in V$. In this case, the associated
unified product $A \ltimes V$ has the multiplication defined for any $a$, $b \in A$ and
$x$, $y \in V$ by:
$$
(a, \, x) \bullet (b, \, y) := \bigl(ab + f(x, y), \,\, a \rightharpoonup y  +
x\triangleleft b \bigl)
$$
that is, $A \ltimes V$ is a cocycle deformation of the usual
trivial extension of $A$ by $V$, dual to the one considered in
\cite{Hoch2}.
\end{example}

Let $\Omega(A, V) = \bigl(\triangleleft, \, \triangleright, \,
\leftharpoonup, \, \rightharpoonup, \, f, \, \cdot \bigl) \in
{\mathcal A} {\mathcal E} (A, V)$ be an algebra extending structure
and $A \ltimes V$ the associated unified product. Then the
canonical inclusion
$$
i_A: A \to A \ltimes V, \qquad i_A (a) = (a, \, 0)
$$
is an injective algebra map. Therefore, we can see $A$ as a
subalgebra of $A \ltimes V$ through the identification $A \cong
i_A (A) = A \times \{0\}$. Conversely, we have the following
result which provides the answer to the description part of the
ES-problem:

\begin{theorem}\thlabel{classif}
Let $A$ be an algebra, $E$ a vector space containing $A$ as a subspace
and $\ast$ an algebra structure on $E$ such that $A$ is a
subalgebra in $(E, \ast)$. Then there exists an algebra
extending structure $\Omega(A, V) = \bigl(\triangleleft, \, \triangleright, \,
\leftharpoonup, \, \rightharpoonup, \, f, \, \cdot \bigl)$ of $A$ through a
subspace $V$ of $E$ and an isomorphism of algebras $(E, \ast ) \cong A \ltimes V$
that stabilizes $A$ and co-stabilizes $V$.
\end{theorem}

\begin{proof}
Since $k$ is a field, there exists a linear map $p: E \to A$ such
that $p(a) = a$, for all $a \in A$. Then $V := \rm{Ker}(p)$ is a
complement of $A$ in $E$. We define the extending datum $\Omega(A,
V) = \bigl(\triangleleft = \triangleleft_p, \, \triangleright =
\triangleright_p, \, \leftharpoonup = \leftharpoonup_p, \,
\rightharpoonup = \rightharpoonup_p, \, f = f_p, \, \cdot =
\cdot_p \bigl)$ of $A$ through $V$ by the following formulas:
\begin{eqnarray*}
&& \triangleright : V \times A \to A, \quad \,\,\,\,\,\,\,
x \triangleright a := p (x \ast a), \quad \,\,\, \triangleleft : V \times A \to V,
\quad x \triangleleft a := x\ast a  - p (x\ast a) \\
&& \leftharpoonup  \, : A \times V \to A, \quad \,\,
a \leftharpoonup x := p(a\ast x), \quad \rightharpoonup \, : A \times V \to V,
\quad  a \rightharpoonup x := a \ast x - p(a\ast x)\\
&& f : V \times V \to A, \quad \,\, f(x, y) := p (x \ast y), \quad
\cdot : V \times V \to V, \qquad x \cdot y := x \ast y - p (x \ast y)
\end{eqnarray*}
for any $a\in A$ and $x$, $y\in V$. We shall prove that $\Omega(A,
V) = \bigl(\triangleleft_p, \, \triangleright_p, \,
\leftharpoonup_p, \, \rightharpoonup_p, \, f_p, \, \cdot_p \bigl)$
is an algebra extending structure of $A$ through $V$ and
\begin{eqnarray*}
\varphi: A \ltimes V \to E, \qquad \varphi(a, x) := a + x
\end{eqnarray*}
is an isomorphism of algebras that stabilizes $A$ and
co-stabilizes $V$. Instead of proving the compatibility conditions
(A1)-(A12), which require a very long and laborious computation,
we use the following trick combined with \thref{1}: $\varphi: A
\times V \to E$, $\varphi(a, \, x) = a + x$ is a linear
isomorphism between the algebra $(E, \ast)$ and the direct product
of vector spaces $A \times V$ with the inverse given by
$\varphi^{-1}(y) := \bigl(p(y), \, y - p(y)\bigl)$, for all $y \in
E$. Thus, there exists a unique algebra structure on $A \times V$
such that $\varphi$ is an isomorphism of algebras and this unique
multiplication $\circ$ on $A \times V$ is given by  $(a, x) \circ
(b, y) := \varphi^{-1} \bigl( \varphi(a, x) \ast \varphi(b,
y)\bigl)$, for all $a$, $b \in A$ and $x$, $y\in V$. The proof is
finished if we prove that this multiplication is the one defined
by \equref{produnif} associated to the system
$\bigl(\triangleleft_p, \, \triangleright_p, \, \leftharpoonup_p,
\, \rightharpoonup_p, \, f_p, \, \cdot_p \bigl)$. Indeed, for any
$a$, $b \in A$ and $x$, $y\in V$ we have:
\begin{eqnarray*}
(a, x) \circ (b, y) &=& \varphi^{-1} \bigl(\varphi(a, x) \ast
\varphi(b, y)\bigl) = \varphi^{-1} \bigl(  (a + x) \ast (b + y) \bigl) \\
&=& \varphi^{-1} ( a b + a \ast y + x \ast b + x \ast y ) \\
&=& \bigl ( ab + p(a \ast y) + p (x \ast b) + p (x \ast y), \\
&& a \ast y - p (a \ast y) + x \ast b - p(x \ast b) + x \ast y - p (x \ast y) \bigl) \\
&=& \bigl( ab + a \leftharpoonup y + x
\triangleright b + f(x, y), \,\, a \rightharpoonup y  +
x\triangleleft b + x \cdot y \bigl) \\
&=& (a, x) \bullet (b, y)
\end{eqnarray*}
as needed. Moreover, the following diagram is commutative
\begin{eqnarray*} \eqlabel{diagrama}
\xymatrix {& A \ar[r]^{i} \ar[d]_{Id} & {A \ltimes V}
\ar[r]^{q} \ar[d]^{\varphi} & V \ar[d]^{Id}\\
& A \ar[r]^{i} & {E}\ar[r]^{\pi } & V}
\end{eqnarray*}
where $\pi : E \to V$ is the projection of $E = A + V$
on the vector space $V$ and $q: A \ltimes V \to V$,
$q (a, x) := x$ is the canonical projection. The proof is now
finished.
\end{proof}

Based on \thref{classif}, the classification of all algebra
structures on $E$ that contain $A$ as a subalgebra reduces to the
classification of all unified products $A \ltimes V$, associated
to all algebra extending structures $\Omega(A, V) =
\bigl(\triangleleft, \, \triangleright, \, \leftharpoonup, \,
\rightharpoonup, \, f, \, \cdot \bigl)$, for a fixed complement
$V$ of $A$ in $E$. Next we will construct explicitly the
non-abelian cohomological type objects ${\mathcal A}{\mathcal
H}^{2}_{A} \, (V, \, A)$ and ${\mathcal A}{\mathcal H}^{2} \, (V,
\, A)$ which will parameterize the classifying sets ${\rm Extd} \,
(E, A)$ and respectively ${\rm Extd}' \, (E, A)$. First we need
the following:

\begin{lemma} \lelabel{morfismuni}
Let $\Omega(A, V) = \bigl(\triangleleft, \, \triangleright, \,
\leftharpoonup, \, \rightharpoonup, \, f, \, \cdot \bigl)$ and
$\Omega(A, V) = \bigl(\triangleleft ', \, \triangleright ', \,
\leftharpoonup ', \, \rightharpoonup ', \, f ', \, \cdot ' \bigl)$
be two algebra extending structures of $A$ through $V$ and $A
\ltimes V$, respectively $A \ltimes ' V$ the associated unified
products. Then there exists a bijection between the set of all
morphisms of algebras $\psi: A \ltimes V \to A \ltimes ' V$which
stabilize $A$ and the set of pairs $(r, \, v)$, where $r: V \to
A$, $v: V \to V$ are linear maps satisfying the following
compatibility conditions for any $a \in A$, $x$, $y \in V$:
\begin{enumerate}
\item[(M1)] $r(x \cdot y) = r(x)r(y) + f ' (v(x),\, v(y)) - f(x,
\, y) + r(x) \leftharpoonup ' v(y) + v(x) \rhd ' r(y)$;
\item[(M2)] $v(x \cdot y) = r(x) \rightharpoonup ' v(y) + v(x)
\lhd ' r(y) + v(x) \cdot ' v(y)$; \item[(M3)] $r(x \lhd a) = r(x)
a - x \rhd a + v(x) \rhd ' a$; \item[(M4)] $v(x \lhd a) = v(x)
\lhd ' a$; \item[(M5)] $r(a \rightharpoonup x) = ar(x) - a
\leftharpoonup x + a \leftharpoonup ' v(x)$; \item[(M6)] $v(a
\rightharpoonup x) = a \rightharpoonup ' v(x)$
\end{enumerate}
Under the above bijection the morphism of algebras $\psi =
\psi_{(r, v)}: A \ltimes V \to A \ltimes ' V$ corresponding to
$(r, v)$ is given by $\psi(a, x) = (a + r(x), v(x))$, for all $a
\in A$ and $x \in V$. Moreover, $\psi = \psi_{(r, v)}$ is an
isomorphism if and only if $v: V \to V$ is an isomorphism and
$\psi = \psi_{(r, v)}$ co-stabilizes $V$ if and only if $v = {\rm
Id}_V$.
\end{lemma}

\begin{proof} For a linear map
$\psi: A \ltimes V \to A \ltimes ' V$ which stabilizes $A$ we have
$\psi(a, 0) = (a, 0)$ for all $a \in A$. Therefore, $\psi$ is
uniquely determined by two linear maps $r: V \to A$, $v: V \to V$
such that $\psi(0, x) = (r(x), v(x))$ for all $x \in V$. Then, for
all $a \in A$ and $x \in V$ we have $\psi(a, x) = (a + r(x),
v(x))$. Let $\psi = \psi_{(r,v)}$ be such a linear map. We will
prove that $\psi$ is an algebra map if and only if the
compatibility conditions (M1)-(M6) hold. It is enough to prove
that the following compatibility holds for all generators of $A
\ltimes V$:
\begin{equation}\eqlabel{algebramap}
\psi\bigl((d, w) \cdot (e, t)\bigl) = \psi\bigl((d, w)\bigl) \cdot
' \psi\bigl((e, t)\bigl)
\end{equation}
By a straightforward computation it follows that
\equref{algebramap} holds for the pair $(a, 0)$, $(b, 0)$ if and
only if (M1) and (M2) are fulfilled while \equref{algebramap}
holds for the pair $(0, x)$, $(a, 0)$ if and only if (M3) and (M4)
are satisfied. Finally, \equref{algebramap} holds for the pair
$(a, 0)$, $(0, x)$ if and only if (M5) and (M6) hold.

Assume now that $v$ is bijective. Then $\psi_{(r, v)}$ is an
isomorphism of algebras with the inverse given by $\psi_{(r,
v)}^{-1} (a,x) = (a - r \bigl(v^{-1}(x)\bigl),\, v^{-1}(x))$ for
all $a \in A$ and $x \in V$. Conversely, assume that
$\psi_{(r,v)}$ is bijective. Then $v$ is obviously surjective.
Consider now $x \in V$ such that $v(x) = 0$. We obtain
$\psi_{(r,v)}(0, 0) = (0, 0) = (0, v(x)) = \psi_{(r, v)}(-r(x),
x)$. As $\psi_{(r,v)}$ is bijective we get $x = 0$. Therefore $v$
is also injective and hence bijective. Finally, it is
straightforward to see that $\psi$ co-stabilizes $V$ if and only
if $v = Id$ and the proof is now finished.
\end{proof}

In order to construct the object that parameterizes ${\rm Extd} \,
(E, A)$ we need the following:

\begin{definition}\delabel{echiaa}
Let $A$ be an algebra and $V$ a vector space. Two algebra
extending structures of $A$ by $V$, $\Omega(A, V) =
\bigl(\triangleleft, \, \triangleright, \, \leftharpoonup, \,
\rightharpoonup, \, f, \, \cdot \bigl)$ and $\Omega(A, V) =
\bigl(\triangleleft ', \, \triangleright ', \, \leftharpoonup ',
\, \rightharpoonup ', \, f ', \, \cdot ' \bigl)$ are called
\emph{equivalent}, and we denote this by $\Omega(A, V) \equiv
\Omega'(A, V)$, if there exists a pair $(r, v)$ of linear maps,
where $r: V \to A$ and $v \in {\rm Aut}_{k}(V)$ such that
$\bigl(\triangleleft ', \, \triangleright ', \, \leftharpoonup ',
\, \rightharpoonup ', \, f ', \, \cdot ' \bigl)$ is implemented
from $\bigl( \triangleleft, \, \triangleright, \, \leftharpoonup,
\, \rightharpoonup, \, f, \, \cdot \bigl)$ using $(r, v)$ via:
\begin{eqnarray*}
a \rightharpoonup ' x &=& v\bigl(a \rightharpoonup v^{-1}(x)\bigl)\\
a \leftharpoonup ' x &=& r\bigl(a \rightharpoonup v^{-1}(x)\bigl)
- a r\bigl(v^{-1}(x)\bigl) + a \leftharpoonup v^{-1}(x)\\
x \lhd ' a &=& v \bigl(v^{-1}(x) \lhd a\bigl)\\
x \rhd ' a &=& r \bigl(v^{-1}(x) \lhd a\bigl) -
r\bigl(v^{-1}(x)\bigl) a + v^{-1}(x) \rhd a\\
x \cdot ' y &=& v\bigl(v^{-1}(x) \cdot v^{-1}(y)\bigl) -
v\Bigl(r\bigl(v^{-1}(x)\bigl) \rightharpoonup v^{-1}(y)\Bigl) -
v\Bigl(v^{-1}(x) \lhd r\bigl(v^{-1}(y)\bigl)\Bigl)\\
f ' (x, y) &=& r\bigl(v^{-1}(x) \cdot v^{-1}(y)\bigl) +
f(v^{-1}(x),\, v^{-1}(y)) - r\Bigl(r\bigl(v^{-1}(x)\bigl)
\rightharpoonup v^{-1}(y)\Bigl) - \\
&& - r\bigl(v^{-1}(x)\bigl) \leftharpoonup v^{-1}(y) -
r\Bigl(v^{-1}(x) \lhd r\bigl(v^{-1}(y)\bigl)\Bigl) +
r\bigl(v^{-1}(x)\bigl) r\bigl(v^{-1}(y)\bigl)-\\ && - v^{-1}(x)
\rhd r\bigl(v^{-1}(y)\bigl)
\end{eqnarray*}
for all $a \in A$, $x$, $y \in V$.
\end{definition}

On the other hand, in order to parameterize ${\rm Extd}' \, (E,
A)$ we need the following:

\begin{definition}\delabel{echiaabb}
Let $A$ be an algebra and $V$ a vector space. Two algebra
extending structures $\Omega(A, V) = \bigl(\triangleleft, \,
\triangleright, \, \leftharpoonup, \, \rightharpoonup, \, f, \,
\cdot \bigl)$ and $\Omega(A, V) = \bigl(\triangleleft ', \,
\triangleright ', \, \leftharpoonup ', \, \rightharpoonup ', \, f
', \, \cdot ' \bigl)$ are called \emph{cohomologous}, and we
denote this by $\Omega(A, V) \approx \Omega'(A, V)$ if and only if
$\triangleleft' = \triangleleft$, $\rightharpoonup' \, = \,
\rightharpoonup$ and there exists a linear map $r: V \to A$ such
that
\begin{eqnarray*}
a \leftharpoonup ' x &=& r(a \rightharpoonup x)
- a r(x) + a \leftharpoonup x\\
x \rhd ' a &=& r(x \lhd a) -
r(x) a + x \rhd a\\
x \cdot ' y &=& x \cdot y - r(x) \rightharpoonup y -
x \lhd r(y)\\
f ' (x, y) &=& r(x \cdot y) + f(x,\, y) - r\Bigl(r(x)
\rightharpoonup y\Bigl)  - r(x) \leftharpoonup y - r\bigl(x \lhd
r(y)\bigl)+\\ &&+ r(x) r(y) - x \rhd r(y)
\end{eqnarray*}
for all $a \in A$, $x$, $y \in V$.
\end{definition}

As a conclusion of this section, the theoretical answer to the
ES-problem now follows:

\begin{theorem}\thlabel{main1}
Let $A$ be an algebra, $E$ a vector space which contains $A$ as a
subspace and $V$ a complement of $A$ in $E$. Then:

$(1)$ $\equiv$ is an equivalence relation on the set ${\mathcal A}
{\mathcal E} (A, V)$ of all algebra extending structures of $A$
through $V$. If we denote by ${\mathcal A}{\mathcal H}^{2}_{A} \,
(V, \, A) := {\mathcal A} {\mathcal E} (A, V)/ \equiv $, then the
map
$$
{\mathcal A}{\mathcal H}^{2}_{A} \, (V, \, A) \to {\rm Extd} \,
(E, A), \qquad \overline{(\triangleleft, \, \triangleright, \,
\leftharpoonup, \, \rightharpoonup, \, f, \, \cdot)} \,
\longmapsto \, \bigl(A \ltimes V, \, \cdot \bigl)
$$
is a bijection between ${\mathcal A}{\mathcal H}^{2}_{A} \, (V, \,
A)$ and the isomorphism classes of all algebra structures on $E$
that contain and stabilize $A$ as a subalgebra.

$(2)$ $\approx $ is an equivalence relation on the set ${\mathcal
A} {\mathcal E} (A, V)$. If we denote by ${\mathcal A}{\mathcal
H}^{2} \, (V, \, A) := {\mathcal A} {\mathcal E} (A, V)/ \approx
$, then the map
$$
{\mathcal A}{\mathcal H}^{2} \, (V, \, A) \to {\rm Extd}' \, (E,
A), \qquad \overline{\overline{(\triangleleft, \, \triangleright,
\, \leftharpoonup, \, \rightharpoonup, \, f, \, \cdot)}} \,
\longmapsto \, \bigl(A \ltimes V, \, \cdot \bigl)
$$
is a bijection between ${\mathcal A}{\mathcal H}^{2} \, (V, \, A)$
and the isomorphism classes of all algebra structures on $E$ which
stabilize $A$ and co-stabilize $V$.\footnote{
$\overline{(\triangleleft, \, \triangleright, \, \leftharpoonup,
\, \rightharpoonup, \, f, \, \cdot)}$ (resp.
$\overline{\overline{(\triangleleft, \, \triangleright, \,
\leftharpoonup, \, \rightharpoonup, \, f, \, \cdot)}}$ denotes the
the equivalence class of $(\triangleleft, \, \triangleright, \,
\leftharpoonup, \, \rightharpoonup, \, f, \, \cdot)$ via $\equiv$
(resp. $\approx$). }
\end{theorem}

\begin{proof} The proof follows
from \thref{1}, \thref{classif} and \leref{morfismuni} once we
observe that $\Omega(\mathfrak{g}, V) \equiv \Omega'(\mathfrak{g},
V)$ in the sense of \deref{echiaa} if and only if there exists an
isomorphism of algebras $\psi: A \ltimes V \to A \ltimes ' V$
which stabilizes $A$. Therefore, $\equiv$ is an equivalence
relation on the set ${\mathcal A} {\mathcal E} (A, V)$ and the
first part follows. In the same way $\Omega(\mathfrak{g}, V)
\approx \Omega'(\mathfrak{g}, V)$ as defined in \deref{echiaabb}
if and only if there exists an isomorphism of algebras $\psi: A
\ltimes V \to A \ltimes ' V$ which stabilizes $A$ and
co-stabilizes $V$ and this proves the second part of the theorem.
\end{proof}

\section{The Galois group of algebra extensions and
special cases of unified products}\selabel{cazurispeciale}

This section contains the first applications of the results
presented in \seref{unifiedprod}. To start with, we will describe
the Galois group ${\rm Gal} \, (B/A)$ of an arbitrary extension $A
\subseteq B$ of non-commutative algebras as a certain subgroup of
a semidirect product of groups ${\rm GL}_k (V) \rtimes {\rm Hom}_k
(V, \, A)$, where the vector space $V$ is a fixed complement of
$A$ in $B$, i.e. $V \cong B/A$ as vector spaces. Then we deal with
several special cases of unified products (explicit examples of
unified products will be provided in \seref{exemple}). We
emphasize the problem for which each of these products is
responsible. We use the following convention: if one of the maps
$\triangleright$, $\leftharpoonup$, $f$ or $\cdot$ of an extending
datum is trivial then we will omit it from the system $\Omega(A,
V) = \bigl(\triangleleft, \, \triangleright, \, \leftharpoonup, \,
\rightharpoonup, \, f, \, \cdot \bigl)$.

\subsection*{The Galois group of an algebra extension}
Let $A \subseteq B$ be an extension of algebras. We define the
\emph{Galois group} ${\rm Gal} \, (B/A)$ of this extension as a
subgroup of ${\rm Aut}_{\rm Alg} (B)$ consisting of all algebra
automorphisms of $B$ that fix $A$, i.e.
$$
{\rm Gal} \, (B/A) := \{ \sigma \in {\rm Aut}_{\rm Alg} (B) \, |
\, \sigma (a) = a, \, \forall \, a\in A \}
$$
For a given subgroup $G \leq {\rm Aut}_{\rm Alg} (B)$ we denote by
$B^G$ its subalgebra of invariants: $B^G := \{b \in B \, | \,
\sigma (b) = b, \, \forall \, \sigma \in G \}$. Of course, we have
that $A \subseteq B^{{\rm Gal} (B/A)}$. As in the classical Galois
theory we say that $A \subseteq B$ is a \emph{Galois extension} if
$B^{{\rm Gal} (B/A)} = A$.

Let $A$ and $V$ be two vector spaces and denote ${\mathbb G}_A^V
:= {\rm Hom}_k (V, \, A) \times {\rm GL}_k (V)$, where ${\rm GL}_k
(V) = {\rm Aut}_k (V)$ is the group of all linear automorphisms of
$V$. Then, we can easily prove that ${\mathbb G}_A^V$ has a group
structure with the multiplication given for any $r$, $r'\in {\rm
Hom}_k (V, \, A)$ and $\sigma$, $\sigma' \in {\rm GL}_k (V)$ by:
\begin{equation} \eqlabel{grupstrc}
(r, \, \sigma) \cdot (r', \, \sigma') := (r' + r \circ \sigma', \,
\sigma \circ \sigma')
\end{equation}
The unit of ${\mathbb G}_A^V$ is $1_{{\mathbb G}_A^V} = (0, \,
{\rm Id}_V)$ and the inverse of $(r, \, \sigma)$ is given by $(r,
\, \sigma)^{-1} := (-r \circ \sigma^{-1}, \, \sigma^{-1})$.
Moreover, one can easily see that ${\rm GL}_k (V) \cong \{0\}
\times {\rm GL}_k (V)$ is a subgroup of ${\mathbb G}_A^V$ and the
abelian group ${\rm Hom}_k (V, \, A) \cong {\rm Hom}_k (V, \, A)
\times \{{\rm Id}_V\}$ is a normal subgroup of ${\mathbb G}_A^V$
since $(r, \, \sigma) \cdot (r', \, {\rm Id}_V) \cdot (r, \,
\sigma)^{-1} = (r' \circ \sigma^{-1}, \, {\rm Id}_V)$, for all
$r$, $r' \in {\rm Hom}_k (V, \, A)$ and $\sigma \in {\rm GL}_k
(V)$. On the other hand, the relation $(r, \, \sigma) = (0, \,
\sigma) \cdot (r, \, {\rm Id}_V)$ gives an exact factorization
${\mathbb G}_A^V = {\rm GL}_k (V) \, \cdot \, {\rm Hom}_k (V, \,
A) $ of the group ${\mathbb G}_A^V$ trough the subgroup ${\rm
GL}_k (V)$ and the abelian normal subgroup ${\rm Hom}_k (V, \,
A)$. These observations show that ${\mathbb G}_A^V$ is a
semidirect product ${\rm GL}_k (V) \rtimes  {\rm Hom}_k (V, \, A)$
of groups, where the semidirect product is written in the right
hand side convention.

Now let $A \subseteq B$ be an extension of algebras. Assume that
the codimension of $A$ in $B$ is $\mathfrak{c}$ and let $V$ be a
vector space of dimension $\mathfrak{c}$ that is a complement of
$A$ in $B$. It follows from \thref{classif} that $B \cong A
\ltimes V$, for a canonical algebra extending structure $\Omega(A,
V) = \bigl(\triangleleft, \, \triangleright, \, \leftharpoonup, \,
\rightharpoonup, \, f, \, \cdot \bigl)$ of $A$ by $V$ associated
to a given linear retraction $p: B \to A$ of the inclusion map $A
\hookrightarrow B$. We fix such an algebra extending structure
$\Omega(A, V) = \bigl(\triangleleft, \, \triangleright, \,
\leftharpoonup, \, \rightharpoonup, \, f, \, \cdot \bigl)$ and we
denote by ${\mathbb G}_A^V \, \bigl(\triangleleft, \,
\triangleright, \, \leftharpoonup, \, \rightharpoonup, \, f, \,
\cdot \bigl)$ be the set of all pairs $(r, \, \sigma) \in {\mathbb
G}_A^V = {\rm Hom}_k (V, \, A) \times {\rm GL}_k (V)$ satisfying
the following compatibility conditions:
\begin{eqnarray*}
&& f(x, \, y) = f(\sigma(x), \, \sigma (y)) + r(x)r(y) - r(x \cdot y) +
r(x) \leftharpoonup  \sigma (y) + \sigma (x) \rhd  r(y) \\
&& \sigma (x \cdot y) = \sigma (x) \cdot \sigma (y) + r(x)
\rightharpoonup  \sigma (y) +
\sigma(x) \lhd  r(y) \\
&& r(x \lhd a) = r(x) a - x \rhd a + \sigma (x) \rhd  a \\
&& r(a \rightharpoonup x) = a\, r(x) - a \leftharpoonup x + a
\leftharpoonup  \sigma (x) \\
&& \sigma (x \lhd a) = \sigma (x) \lhd  a, \qquad \sigma (a
\rightharpoonup x) = a \rightharpoonup  \sigma (x)
\end{eqnarray*}
for all $a \in A$, $x$, $y \in V$. Using \thref{classif} and
\leref{morfismuni} we obtain an explicit description for the
Galois group of an algebra extension.

\begin{corollary} \colabel{galgrup}
Let $A \subseteq B$ be an extension of algebras. We fix $V$ a
vector space and $\Omega(A, V) = \bigl(\triangleleft, \,
\triangleright, \, \leftharpoonup, \, \rightharpoonup, \, f, \,
\cdot \bigl)$ an algebra extending structure of $A$ by $V$ such
that $B \cong A \ltimes V$. Then, ${\mathbb G}_A^V \,
\bigl(\triangleleft, \, \triangleright, \, \leftharpoonup, \,
\rightharpoonup, \, f, \, \cdot \bigl)$ is a subgroup of the
semidirect product of groups ${\mathbb G}_A^V = {\rm GL}_k (V)
\rtimes {\rm Hom}_k (V, \, A)$ and there exists an isomorphism of
groups ${\rm Gal} (B/A) \cong {\mathbb G}_A^V \,
\bigl(\triangleleft, \, \triangleright, \, \leftharpoonup, \,
\rightharpoonup, \, f, \, \cdot \bigl)$.
\end{corollary}

\begin{proof}
Indeed, it follows from \leref{morfismuni} that there exists a
bijective correspondence between elements of ${\rm Gal} \,
(A\ltimes V /A)$ and the set ${\mathbb G}_A^V \,
\bigl(\triangleleft, \, \triangleright, \, \leftharpoonup, \,
\rightharpoonup, \, f, \, \cdot \bigl)$ defined such that
$\psi_{(r, \sigma)} \in {\rm Gal} \, (A\ltimes V /A)$
corresponding to $(r, \sigma)$ is given by $\psi_{(r, \, \sigma)}
(a, x) = (a + r(x), \, \sigma(x))$, for all $a \in A$ and $x \in
V$. The proof is finished once we observe that $\psi_{(r, \,
\sigma)} \circ \psi_{(r', \, \sigma')} = \psi_{(r' + r \circ
\sigma', \, \sigma \circ \sigma')}$, for all $(r, \, \sigma)$,
$(r', \, \sigma') \in {\mathbb G}_A^V \, \bigl(\triangleleft, \,
\triangleright, \, \leftharpoonup, \, \rightharpoonup, \, f, \,
\cdot \bigl)$.
\end{proof}

The Galois group ${\mathbb G}_A^V \, \bigl(\triangleleft, \,
\triangleright, \, \leftharpoonup, \, \rightharpoonup, \, f, \,
\cdot \bigl)$ has an abelian normal subgroup corresponding to
those automorphisms of $A \ltimes V$ which stabilize $A$ and
co-stabilize $V$. We denote this subgroup by ${\mathbb H}_A^V \,
\bigl(\triangleleft, \, \triangleright, \, \leftharpoonup, \,
\rightharpoonup, \, f, \, \cdot \bigl)$ -- it can be identified
with the set of all linear maps $r \in {\rm Hom}_k (V, \, A)$ such
that for any $a \in A$, $x$, $y \in V$ :
\begin{eqnarray*}
&& r(x \cdot y) = r(x)r(y) + r(x) \leftharpoonup  y + x \rhd r(y),
\quad r(x) \rightharpoonup y = -  x \lhd  r(y) \\
&& r(x \lhd a) = r(x) a, \qquad  r(a \rightharpoonup x) = a\, r(x)
\end{eqnarray*}

In the finite dimensional case we obtain:

\begin{corollary} \colabel{galgrupdim}
Let $A \subseteq B$ be an extension of algebras such that ${\rm
dim}_k (A) = n$ and ${\rm dim}_k (B) = n + m$, for two positive
integers $m$ and $n$. Then there exists a canonical embedding
${\rm Gal} \, (B/A) \hookrightarrow {\rm GL} (m, \, k) \rtimes
{\rm M}_{n\times m} (k)$, where ${\rm M}_{n\times m} (k)$ is the
additive group of $(n\times m)$-matrices over $k$.
\end{corollary}

Thus, in the light of \coref{galgrupdim}, in order to compute the
Galois groups for extensions of finite dimensional algebras, the
groups ${\rm GL} (m, \, k) \rtimes {\rm M}_{n\times m} (k)$ will
play the role of the symmetric group $S_n$ from the classical
theory. A more detailed investigation of these concepts will be
performed somewhere else, here we only indicate a few examples and
in \coref{galcodim1} we compute the Galois group for any algebra
extension $A \subseteq B$ such that $A$ has codimension $1$ in
$B$.

\begin{examples}\exlabel{anagal}
1. Let $A := k < x \, | \, x^2 = 0 > $ and $B$ be the
$3$-dimensional non-commutative algebra having $\{1, \, x, \, y\}$
as a basis and the multiplication defined by $x^2 = 0$, $y^2 = y$,
$xy = x$ and $yx = 0$. Then ${\rm Gal} \, (B/A) \cong (k, +) $ and
$A \subseteq B$ is a Galois extension of algebras.

2. Consider  $B$ to be the non-commutative algebra having $\{1, \,
x, \, y\}$ as a basis and the multiplication defined by $x^2 = x$,
$y^2 = 0$, $xy = x$ and $yx = 0$ . Then  ${\rm Gal} \, (B/k) \cong
(k, +) \ltimes (k^{*}, \cdot) $, where we denoted by $ (k, +)
\ltimes (k^{*}, \cdot) $ the semidirect product of groups
associated to the group homomorphism $\psi: (k^{*}, \cdot) \to
{\rm Aut}( (k, +) )$ given by $\psi(h)(t) = ht$ for all $h \in
k^{*}$, $t \in k$. Furthermore, the extension of algebras $k
\subset B$ is Galois.

3. Let $A \subseteq B$ be the extension of algebras where $B$ is
the $4$-dimensional algebra having $\{e_1, \, e_2, \, e_3, e_4 \}$
as a basis and the multiplication given by $e_1 e_1 = e_1$, $e_1
e_3 = e_3 e_1 = e_3$, $e_2 e_2 = e_2$, $e_2 e_4 = e_4$ and $e_4
e_1 = e_4$ (undefined operations are zero) and $A$ is the
subalgebra with $\{e_1, \, e_2\}$ as a basis. The unit is $1_B =
e_1 + e_2$. Then, we can easily find that there exists an
isomorphism of groups ${\rm Gal} \, (B/A) \cong k^* \times k^* $
given such that $\sigma \in {\rm Gal} \, (B/A)$ associated to $(a,
\, b) \in k^* \times k^* $ is given by $\sigma (e_1) = e_1$,
$\sigma (e_2) = e_2$, $\sigma (e_3) = a e_3$ and $\sigma (e_4) = b
e_4$. Moreover, the extension $A \subseteq B$ in Galois, since $A
= B^{{\rm Gal} (B/A)}$.

Moreover, if we consider the same algebra $A$ as above but this
time viewed as a subalgebra of the $4$-dimensional algebra $C$,
having the basis $\{e_1, \, e_2, \, e_3, e_4 \}$ and
multiplication $e_1 e_1 = e_1$, $e_2 e_2 = e_2$, $e_2 e_3 = e_3$,
$e_2 e_4 = e_4$, $e_3 e_1 = e_3$ and $e_4 e_1 = e_4$ then we will
obtain that there exists an isomorphism of groups ${\rm Gal} \,
(C/A) \cong {\rm GL} (2, \, k) $ given such that $\sigma \in {\rm
Gal} \, (B/A)$ associated to the $2\times 2$ invertible matrix
$(a_{ij})$ is given by $\sigma (e_3) = a_{11} e_3 + a_{21} e_4$
and $\sigma (e_4) = a_{12} e_3 + a_{22} e_4$. The extension $A
\subseteq C$ is also Galois, as shown by an elementary
computation.
\end{examples}

\subsection*{Relative split extensions and cocycle semidirect products of algebras}
We will prove that several special cases of the unified product
are responsible for the description of algebra extensions $A
\subset E$ which split as morphisms of left/right $A$-modules,
$A$-bimodule or as algebra maps. If $A \subset E$ is an inclusion
of algebras, then $E$ will be viewed as a left/right $A$-module
via the restriction of scalars: $a \cdot x \cdot a' := a x a'$,
for all $a$, $a'\in A$ and $x\in E$. If $(V, \rightharpoonup,
\triangleleft) \in {}_A\Mm_A$ is an $A$-bimodule, then $V \times
V$ is viewed as an $A$-bimodule in the canonical way, i.e. the
left (resp. right) action of $A$ on $V\times V$ is implemented by
$\rightharpoonup$ (resp. $\triangleleft$). A bilinear map $f: V
\times V \to A$ is called $A$-balanced if $f (x, \, a
\rightharpoonup y) = f (x \triangleleft a, \, y)$, for all $a\in
A$, $x$, $y\in V$. Of course, $A$-bimodules and $A$-balanced maps
$f : V \times V \to A$ are in bijection to the set of all
$A$-bimodule maps $\tilde{f} : V \ot_A V \to A$, where $\ot_A$ is
the tensor product over $A$.

First, we shall describe extensions of algebras $A \subset E$ that
split in ${}_A\Mm$ (resp. $\Mm_A$), i.e. there exists a left
(resp. right) $A$-module map $p : E \to A$ such that $p (a) = a$,
for all $a\in A$.

\begin{corollary} \colabel{splitleft}
An extension of algebras $A \subset E$ has a retraction that is a
left (resp. right) $A$-module map if and only if there exists an
isomorphism of algebras $E \cong A \ltimes V$, where $A \ltimes V$
is the unified product associated to an algebra extending
structure $\Omega(A, V) \in {\mathcal A} {\mathcal E} (A, V) $
having $ \leftharpoonup : A \times V \to A$ (resp. $\triangleright
: V \times A \to A$) the trivial map.
\end{corollary}

\begin{proof}
Let $A \ltimes V$ be a unified product associated to $\Omega(A, V) \in
{\mathcal A} {\mathcal E} (A, V) $, for which $ \leftharpoonup$ is the trivial map. Then
the canonical projection $p_A : A \ltimes V \to A$ is a retraction of the inclusion
$i_A : A \to A\ltimes V$ that is also left $A$-linear since, using \equref{001}, we have:
$$
p_A (a \cdot (b, x) ) = p_A \bigl( (a, 0) \bullet (b, x) \bigl )
= p_A (ab, \, a \rightharpoonup x) = ab = a p_A ( (b, x))
$$
for all $a$, $b\in A$, $x \in V$. Conversely, let $A \subset E$ be
an inclusion of algebras which has a retraction $p: E \to A$ that
is also a left $A$-module map. It follows from the proof of
\thref{classif} that the action $\leftharpoonup = \leftharpoonup_p
$ associated to the retraction $p$ is the trivial map since $ a
\leftharpoonup x = p(a x) = a p (x) = 0 $, for all $a \in A$ and
$x \in V = {\rm Ker}(p)$. Thus, there exists an isomorphism of
algebras $E \cong A \ltimes V$, where $A \ltimes V$ is the unified
product associated to $\Omega(A, V) \in {\mathcal A} {\mathcal E}
(A, V) $, for which $ \leftharpoonup$ is the trivial map.

Analogously we can prove that the algebra extensions $A \subset E$
that split as right $A$-module maps are parameterized by the
unified products $A \ltimes V$ associated to algebra extending
structures $\Omega(A, V) \in {\mathcal A} {\mathcal E} (A, V)$ for
which the action $\triangleright : V \times A \to A$ is the
trivial map.
\end{proof}

\begin{examples} \exlabel{exgrupalecros}
1. The basic example of an algebra extension which splits as in
\coref{splitleft} is a group algebras extension. Let $H \leq G$ be
a subgroup of a group $G$. Then the group algebras extension $k[H]
\subset k[G]$ has a retraction which is a left $k[H]$-module map.
Hence, there exists an isomorphism of algebras $k[G] \cong k[H]
\ltimes V$, for some algebra extending structure $\Omega(k[H], V)
\in {\mathcal A} {\mathcal E} (k[H], V) $ having $ \leftharpoonup
: k[H] \times V \to A$ the trivial map.

2. The second class of split extensions in the sense of
\coref{splitleft} are the classical crossed products of algebras
\cite{passman}. Let $A$ be an algebra, $G$ be a group and $\alpha:
G \to {\rm Aut} (A)$, $f: G\times G \to U(A)$ be two maps. We
shall denote by $g \triangleright a := \alpha (g) (a)$, for all $g
\in G$ and $a\in A$. Let $\overline{G}$ be a copy as a set of the
group $G$ and $A^f_{\alpha}[G]$ be the free left $A$-module having
$\overline{G}$ as an $A$-basis with the multiplication given by:
\begin{equation} \eqlabel{multcross}
(a \, \overline{g}) (b \, \overline{h}) := a (g \triangleright b) f(g, h) \, \overline{gh}
\end{equation}
for all $a$, $b\in A$ and $g$, $h\in G$. $A^f_{\alpha}[G]$ is
called the \emph{crossed product} of $A$ and $G$ if it is an
associative algebra with the unit $1_A \overline{1_G}$. This is
equivalent (\cite{passman}) to the fact that $f(1_G, 1_G) =1_A$
and the following compatibilities hold for any $g$, $h$, $l\in G$
and $a \in A$:
\begin{eqnarray*}
g \triangleright ( h \triangleright a) = f(g, h) \bigl( (gh)\triangleright a \bigl) f(g, h)^{-1}, \qquad
f (g, h) f(gh, l) = \bigl( g \triangleright f(h, l) \bigl) f(g, hl);
\end{eqnarray*}
Any crossed product $A^f_{\alpha}[G]$ is an extension of $A$ via
the canonical map $i_A : A \to A^f_{\alpha}[G]$, $i_A (a) := a
\overline{1_G}$. This extension splits in the category of left
$A$-modules: the left $A$-linear map that splits $i_A$ being the
augmentation map $\pi_A : A^f_{\alpha}[G] \to A$, $\pi_A (a
\overline{g}) := a$, for all $a\in A$ and $g\in G$. Thus, using
\coref{splitleft} we obtain that any crossed product
$A^f_{\alpha}[G]$ is isomorphic to a unified product $A \ltimes V$
associated to an algebra extending structure $\Omega(A, V) \in
{\mathcal A} {\mathcal E} (A, V)$ for which the action
$\leftharpoonup : A \times V \to A$  is the trivial map.

3. Let $A$ be an algebra and $(W, 1_W)$ a pointed vector space.
All algebra structures $\cdot$ on the vector space $A \ot W$ such
that $(a \ot 1_W) \cdot (b \ot w) = ab \ot w$, for all $a$, $b\in
A$ and $w\in W$ and having $1_A \ot 1_W$ as a unit are fully
described in \cite[Proposition 2.1]{TB}: they are parameterized by
the set of all pairs $(\sigma, R)$ consisting of two linear maps
$\sigma: W\ot W \to A\ot W$, $R: W\ot A \to A \ot W$, satisfying a
laborious set of axioms. Such an algebra structure, which is a
very general construction, is denoted by $A \otimes_{R, \sigma} W$
and is called the \emph{Brzezinski's product}; they generalize the
twisted tensor product algebras \cite{cap} and are classified, up
to an isomorphism of algebras that stabilizes $A$, in
\cite{panaite2}. Now, $i_A : A \to A \otimes_{R, \sigma} W $, $i_A
(a) = a \ot 1_W$ is an injective algebra map which has a
retraction that is a left $A$-module map. Indeed, let $B = \{e_i
\, | \, i \in I\}$ be a basis in $W$ such that $1_W \in B$ and
$\varepsilon: W \to k$, $\varepsilon (e_i) = 1$, for all $i\in I$.
Then $\pi_A : A \otimes_{R, \sigma} W \to A$, $\pi_A := {\rm Id}_A
\ot \varepsilon$ is a left $A$-linear map and a retraction of
$i_A$. Thus, $ A \otimes_{R, \sigma} W \cong A \ltimes V$ for an
algebra extending structure $\Omega(A, V) \in {\mathcal A}
{\mathcal E} (A, V)$ for which the action $\leftharpoonup : A
\times V \to A$  is the trivial map.

4. The Ore extensions are also a special case of the unified
product; in particular any Weyl algebra is a unified product.
Indeed, let $\sigma: A \to A$ be an automorphism of the algebra
$A$, $\delta :A \to A$ a $\sigma$-derivation and $A[X, \sigma,
\delta]$ the Ore extension associated to $(\sigma, \delta)$, that
is $A[X, \sigma, \delta]$ is the free left $A$-module having
$\{X^n \, | \, n \geq 0 \}$ as a basis and the multiplication
given by $X a = \sigma (a) X + \delta(a)$, for all $a\in A$. Then,
the canonical embedding $i_A: A \to A[X, \sigma, \delta]$, $i_A
(a) = a$, has a retraction $p_A : A[X, \sigma, \delta] \to A$ that
is a left $A$-module map given by $p_A \bigl( \sum_{i=0}^n \, a_i
X^i \bigl) := a_0$. Thus, there exists an isomorphism of algebras
$A[X, \sigma, \delta] \cong A \ltimes V$, for an algebra extending
structure $\Omega(A, V) \in {\mathcal A} {\mathcal E} (A, V)$
having $\leftharpoonup : A \times V \to A$ the trivial map.
\end{examples}

Using \coref{splitleft} we can describe the algebra extensions $A
\subset E$ that admit a retraction $p: E \to A$ which is an
$A$-bimodule map. In this case the axioms (A1)- (A12) which
describe the corresponding unified products simplify considerably.
Indeed, let $\Omega(A, V) = \bigl(\triangleleft, \,
\triangleright, \, \leftharpoonup, \, \rightharpoonup, \, f, \,
\cdot \bigl)$ be an extending datum such that $\leftharpoonup$ and
$\triangleright$ are both the trivial maps. Then $\Omega(A, V) =
\bigl(\triangleleft, \, \rightharpoonup, \, f, \, \cdot \bigl)$ is
an algebra extending structure of $A$ through $V$ if and only if
$(V, \rightharpoonup, \triangleleft) \in {}_A\Mm_A$ is an
$A$-bimodule, $f: V\times V \to A$ is an $A$-balanced $A$-bimodule
map and the following compatibilities hold for any $a$, $b \in A$,
$x$, $y$, $z \in V$:
\begin{eqnarray*}
&& x \cdot ( y \cdot z) - (x \cdot y) \cdot z = f (x, \,
y) \rightharpoonup z \, - \,  x \triangleleft f (y, \, z) \\
&& a \rightharpoonup (x \cdot y) = (a \rightharpoonup x)
\cdot y, \qquad (x\cdot y) \triangleleft a = x \cdot (y \triangleleft a) \\
&& x \cdot (a \rightharpoonup y) = (x \triangleleft a) \cdot y, \qquad \,\,
f(x, \, y\cdot z) = f(x\cdot y, \, z)
\end{eqnarray*}
A system $\Omega(A, V) = \bigl(\triangleleft, \, \rightharpoonup,
\, f, \, \cdot \bigl)$ satisfying these compatibilities will be
called a \emph{cocycle semidirect system} of algebras. The unified
product associated to a cocycle semidirect system $\Omega(A, V) =
\bigl(\triangleleft, \, \rightharpoonup, \, f, \, \cdot \bigl)$
will be denoted by $A \#^f V$ and will be called the \emph{cocycle
semidirect product of algebras}. Explicitly, $A \#^f V = A \times
V$ (as vector spaces) with the multiplication given by:
\begin{equation}\eqlabel{cocsemidi}
(a, \, x) \bullet (b, \, y) := \bigl( ab + f(x, y), \,\, a \rightharpoonup y  +
x\triangleleft b + x \cdot y \bigl)
\end{equation}
for all $a$, $b \in A$ and $x$, $y \in V$. \coref{splitleft}
provides the following result:

\begin{corollary} \colabel{splitbimo}
An extension of algebras $A \subset E$ has a retraction that is an
$A$-bimodule map if and only if there exists an isomorphism of
algebras $E \cong A \#^f V$, where $A \#^f V$ is a cocycle
semidirect product of algebras.
\end{corollary}

\begin{example} \exlabel{twisted}
Examples of algebras that split in the sense of \coref{splitbimo}
are the \emph{twisted products} of algebras. A twisted product is
a crossed product $A^f_{\alpha}[G]$ as defined in
\exref{exgrupalecros} for which the action $\alpha$ is the trivial
action, that is $g \triangleright a = a$, for all $g\in G$ and
$a\in A$. In this case the augmentation map $\pi_A$ is also a
right $A$-module map and thus any twisted product $A^f[G]$ is
isomorphic to a cocycle semidirect product of algebras.
\end{example}

A cocycle semidirect system of algebras $\Omega(A, V) =
\bigl(\triangleleft, \, \rightharpoonup, \, f, \, \cdot \bigl)$
for which $f$ is the trivial map is called a \emph{semidirect
system} of algebras. Explicitly, $\Omega(A, V) =
\bigl(\triangleleft, \, \rightharpoonup, \, \cdot \bigl)$ is a
semidirect system of algebras if and only if $(V, \rightharpoonup,
\triangleleft) \in {}_A\Mm_A$ is an $A$-bimodule, $(V, \cdot)$ is
an associative (not-necessarily unital) algebra and
\begin{eqnarray*}
a \rightharpoonup (x \cdot y) = (a \rightharpoonup x)
\cdot y, \quad (x\cdot y) \triangleleft a = x \cdot (y \triangleleft a), \quad
x \cdot (a \rightharpoonup y) = (x \triangleleft a) \cdot y
\end{eqnarray*}
for all $a\in A$, $x$, $y\in V$. The cocycle semidirect product of
algebras associated to a semidirect system $\Omega(A, V) =
\bigl(\triangleleft, \, \rightharpoonup, \, \cdot \bigl)$ is
called the \emph{semidirect product} of algebras and will be
denoted by $A \# V$. This is a classical construction: it appears
in an equivalent form in \cite[Lemma a, pg. 212]{pierce}. We call
it the semidirect product by analogy with the group and Lie
algebra case where the semidirect product describes the split
extensions. More precisely, we have:

\begin{corollary} \colabel{splialg}
An extension of algebras $A \subset E$ has a retraction that is an algebra map if and only if
there exists an isomorphism of algebras $E \cong A \# V$, where $A \# V$ is
a semidirect product of algebras.
\end{corollary}

\begin{proof}
Indeed, the canonical projection $p_A : A \# V \to A$, $p_{A}(a,
\, x) = a$ is a retraction of the inclusion $i_A : A \to A\# V$
and an algebra map. Conversely, from \thref{classif} it follows
that if $p: E \to A$ is an algebra map then $\triangleright_p$,
$\leftharpoonup_p$ and $f_p$ constructed in the proof are all
trivial maps, i.e. the corresponding unified product $A\ltimes V$
is a semidirect product $A\#V$.
\end{proof}

\subsection*{Matched pairs, bicrossed products and the factorization problem for
algebras}

The concept of a matched pair of groups was introduced in
\cite{Takeuchi} while the one for Lie algebras in \cite[Theorem
4.1]{majid} and independently in \cite[Theorem 3.9]{LW}. To any
matched pair of groups (resp. Lie algebras) a new group (resp. Lie
algebra) called the \emph{bicrossed product} is associated and it
is responsible for the so-called factorization problem. Since
then, the corresponding concepts were introduced for several types
of categories such as groupoids, Hopf algebras, local compact
quantum groups, etc. In what follows we will introduce the
corresponding notion for associative algebras. First, we set the
terminology. If $(V, \cdot)$ is an associative (not-necessarily
unital) algebra, then the concept of left/right $V$-module or
$V$-bimodule is defined as in the case of unital algebras except
of course for the unitary condition.

\begin{definition} \delabel{mpalgebras}
A \emph{matched pair} of algebras is a system $(A, \, V, \,
\triangleleft, \, \triangleright, \, \leftharpoonup, \,
\rightharpoonup \bigl)$ consisting of an unital algebra $A$, a
(not-necessarily unital) associative algebra $V = (V, \cdot) $ and
four bilinear maps
$$
\triangleleft : V \times A \to V, \quad \triangleright : V \times
A \to A, \quad \leftharpoonup \, : A \times V \to A, \quad
\rightharpoonup \, : A \times V \to V
$$
such that $(V, \rightharpoonup, \triangleleft) \in {}_A\Mm_A$ is
an $A$-bimodule, $(A, \triangleright, \leftharpoonup) \in
{}_V\Mm_V$ is a $V$-bimodule and the following compatibilities
hold for any $a$, $b \in A$, $x$, $y \in V$:
\begin{enumerate}
\item[(MP1)] $a \rightharpoonup (x \cdot y) = (a \rightharpoonup x)
\cdot y + (a \leftharpoonup x) \rightharpoonup y$;

\item[(MP2)] $ (ab) \leftharpoonup x = a (b \leftharpoonup x) + a
\leftharpoonup (b \rightharpoonup x)$;

\item[(MP3)] $ x \triangleright (ab) = (x \triangleright a) b + (x
\triangleleft a) \triangleright b$;

\item[(MP4)] $(x\cdot y) \triangleleft a = x \triangleleft (y
\triangleright  a) + x \cdot (y \triangleleft a)$;

\item[(MP5)] $a (x \triangleright b) + a \leftharpoonup (x
\triangleleft b) = (a \leftharpoonup x) b + (a\rightharpoonup x)
\triangleright b$;

\item[(MP6)] $x \triangleleft (a \leftharpoonup y) + x \cdot (a
\rightharpoonup y) = (x \triangleright a) \rightharpoonup y + (x
\triangleleft a) \cdot y $;
\end{enumerate}
\end{definition}

We make a few comments on these compatibilities. If we apply (MP2)
and (MP3) for $a = b = 1_A$, we obtain that $1_A \leftharpoonup x
= 0$ and $x \triangleright 1_A = 0$, for all $x\in V$. The two
compatibility conditions together with the unitary condition
derived from the fact that $(V, \rightharpoonup, \triangleleft)$
is an $A$-bimodule show that the system $(\triangleleft, \,
\triangleright, \, \leftharpoonup, \, \rightharpoonup \bigl)$ is
normalized in the sense of \deref{exdatum}. Similar to the Lie
algebra \cite{am-2013} case the above axioms can be derived from
the ones of an algebra extending structure for which the cocycle
$f$ is the trivial map. More precisely, let $\Omega(A, V) =
\bigl(\triangleleft, \, \triangleright, \, \leftharpoonup, \,
\rightharpoonup, \, f, \, \cdot \bigl)$ be an extending datum such
that $f$ is the trivial map. Then $\Omega(A, V) =
\bigl(\triangleleft, \, \triangleright, \, \leftharpoonup, \,
\rightharpoonup, \, \cdot \bigl)$ is an algebra extending
structure of $A$ through $V$ if and only if $(A, \, (V, \cdot), \,
\triangleleft, \, \triangleright, \, \leftharpoonup, \,
\rightharpoonup \bigl)$ is a matched pair of algebras. In this
case, the associated unified product $A \ltimes V$ will be
denoted, as in the case of groups, Lie algebras, Hopf algebras,
etc. by $A \bowtie V$ and will be called the \emph{bicrossed
product} associated to the matched pair $(A, \, V, \,
\triangleleft, \, \triangleright, \, \leftharpoonup, \,
\rightharpoonup \bigl)$. Thus, $A \bowtie V = A \times V$, as a
vector space, with an unital associative algebra structure given
by
\begin{equation}\eqlabel{prodbicross}
(a, \, x) \bullet (b, \, y) := \bigl( ab + a \leftharpoonup y + x
\triangleright b, \,\, a \rightharpoonup y  +
x\triangleleft b + x \cdot y \bigl)
\end{equation}
for all $a$, $b \in A$ and $x$, $y \in V$. The bicrossed product
of two algebras is the construction responsible for the so-called
\emph{factorization problem}, which in the case of associative
algebras comes down to:

\emph{Let $A$ be a unital algebra and $V$ a (not-necessarily
unital) associative algebra. Describe and classify all unital
algebras $E$ that factorize through $A$ and $V$, i.e. $E$ contains
$A$ and $V$ as subalgebras such that $E = A + V$ and $A \cap V =
\{0\}$.}

Indeed, as a special case of \thref{classif} we have:

\begin{corollary}\colabel{bicrfactor}
Let $A$ be an unital algebra and $V$ a (not-necessarily unital)
associative algebra. Then, an algebra $E$ factorizes through $A$
and $V$ if and only if there exists a matched pair of algebras
$(A, \, V, \, \triangleleft, \, \triangleright, \, \leftharpoonup,
\, \rightharpoonup \bigl)$ such that $ E \cong A \bowtie V$.
\end{corollary}

\begin{proof}
To start with, it is easy to see that any bicrossed product $A
\bowtie V$ factorizes through $A \cong A \times \{0\}$ and $V
\cong \{0\}\times V$, which are subalgebras in $A \bowtie V$.
Conversely, assume that $E$ factorizes through $A$ and $V$. Let
$p: E \to A$ be the $k$-linear projection of $E$ on $A$, i.e. $p
(a + x) := a$, for all $a\in A$ and $x \in V$. Now, we apply
\thref{classif} for $V = {\rm Ker}(p)$. Since $V$ is a subalgebra
of $E$, the map $f = f_p$ constructed in the proof of
\thref{classif} is the trivial map as $x \cdot y \in V = {\rm Ker}
(p)$. Thus, the algebra extending structure $\Omega(A, V)$
constructed in the proof of \thref{classif} is precisely a matched
pair of algebras and the unified product $A \ltimes V = A \bowtie
V$ is the bicrossed product of the matched pair  $(A, \, V, \,
\triangleleft, \, \triangleright, \, \leftharpoonup, \,
\rightharpoonup \bigl)$.
\end{proof}

\coref{bicrfactor} shows that the factorization problem for
algebras can be restated in a purely computational manner as
follows: Let $A$ and $V$ be two given algebras. Describe the set
of all matched pairs of algebras $(A, \, V, \, \triangleleft, \,
\triangleright, \, \leftharpoonup, \, \rightharpoonup \bigl)$ and
classify up to an isomorphism all bicrossed products $A \bowtie
V$. A detailed study of this question will be given elsewhere.

\subsection*{The commutative case}
The case of commutative algebras needs to be treated distinctly.
On the one hand, we obtain from \equref{001a} and \equref{002b}
that a unified product $A \ltimes V$ is a commutative algebra if
and only if $A$ is commutative, $f: V \times V \to A$, $\cdot : V
\times V \to V$ are symmetric bilinear maps, $a \leftharpoonup x =
x \triangleright a$ and $a\rightharpoonup x  = x \triangleleft a$,
for all $a\in A$ and $x \in V$. On the other hand, if we look at
the construction of the algebra extending structure from
\thref{classif} in the case when $A\subseteq E$ is an extension of
commutative algebras, we also obtain that $a \leftharpoonup_p x =
x \triangleright_p a$ and $a\rightharpoonup_p x  = x
\triangleleft_p a$ for all $a\in A$ and $x \in V$. Thus, in the
commutative case \deref{exdatum} takes the following form:

\begin{definition}\delabel{comexdatum}
Let $A$ be a commutative algebra and $V$ a vector space. A
\textit{commutative extending datum of $A$ through $V$} is a
system $\Omega(A, V) = \bigl(\triangleleft, \, \triangleright, \,
f, \, \cdot \bigl)$ consisting of four bilinear maps
$$
\triangleleft : V \times A \to V, \quad \triangleright : V \times
A \to A, \quad f: V\times V \to A, \quad \cdot \, : V\times V \to V
$$
such that $f$ and $\cdot$ are symmetric. Let $\Omega(A, V) =
\bigl(\triangleleft, \, \triangleright, \, f, \, \cdot \bigl)$ be
a commutative extending datum. Then the multiplication on $A \,
\ltimes V = A \, \times V $ given by \equref{produnif} takes the
form:
\begin{equation}\eqlabel{produnifcom}
(a, \, x) \bullet (b, \, y) := \bigl( ab + y \triangleright a + x
\triangleright b + f(x, y), \,\,  y \triangleleft a  +
x\triangleleft b + x \cdot y \bigl)
\end{equation}
for all $a$, $b \in A$ and $x$, $y \in V$. $A\ltimes V$
is a commutative unified product if it is a commutative
associative algebra with the multiplication given by
\equref{produnifcom} and the unit $(1_A, 0_V)$. In this case the
extending datum $\Omega(A, V) = \bigl(\triangleleft, \,
\triangleright, \, f, \, \cdot \bigl)$ is called a \textit{commutative
algebra extending structure} of $A$ through $V$.
\end{definition}

In other words, a commutative algebra extending structure of a
commutative algebra $A$ through a vector space $V$ is a
commutative extending datum $\Omega(A, V) = \bigl(\triangleleft,
\, \triangleright, \, f, \, \cdot \bigl)$ satisfying the axioms
(A1)-(A12) of \thref{1} in which we replace $a \leftharpoonup x :=
x \triangleright a$ and $a\rightharpoonup x := x \triangleleft a$
for all $a\in A$, $x \in V$. That is, the following compatibility
conditions hold for any $a$, $b \in A$, $x$, $y$, $z \in V$:
\begin{enumerate}
\item[(CA1)] $(V, \triangleleft)$ is an $A$-module and $x \triangleright 1_A = 0$;

\item[(CA2)] $x \cdot ( y \cdot z) - (x \cdot y) \cdot z = z \triangleleft f (x, \,
y) \, - \,  x \triangleleft f (y, \, z)$;

\item[(CA3)] $(x \cdot y) \triangleleft a = x \triangleleft (y
\triangleright a) + x \cdot (y\triangleleft a)$;

\item[(CA4)] $ x \triangleright (ab) = a (x \triangleright b) + (x
\triangleleft b) \triangleright a$;

\item[(CA5)] $ (x \cdot y) \triangleright a = x \triangleright (y
\triangleright a) + f(x, \, y \triangleleft a) - f (x, \, y) a$;

\item[(CA6)] $f(x, \, y\cdot z) - f(x\cdot y, \, z) = z
\triangleright f (x, \, y)
 - x \triangleright f(y, \, z)$;
\end{enumerate}

The above axioms are derived from those of \thref{1} by taking
into account that $A$ is commutative, $f$ and $\cdot$ are
symmetric bilinear maps and $a \leftharpoonup x = x \triangleright
a$ and $a\rightharpoonup x  = x \triangleleft a$ for all $a\in A$
and $x \in V$. Indeed, as $a \leftharpoonup x = x \triangleright
a$ and $a\rightharpoonup x  = x \triangleleft a$, the normalizing
conditions come down to $x \rhd 1_{A} = 0$ while (A1) reduces to
$(V, \lhd)$ being an $A$-module. Moreover, (A4) follows from (A7)
by using $(x \cdot y) \rhd a = (y \cdot x) \rhd a$. In the same
manner we can derive (A8) out of (A3) by having in mind that $(x
\cdot y) \lhd a = (y \cdot x) \lhd a$. (A6) and (A9) can be
derived from (A5) by using the commutativity of $A$, more
precisely $x \rhd ab = x \rhd ba$. Finally, (A10) comes out of
(A4) by having in mind that $(x \cdot y) \rhd a = (y \cdot x) \rhd
a$ while (A11) follows from (A3) by using $(x \cdot y) \lhd a = (y
\cdot x) \lhd a$. Therefore, we are left with the independent set
of $6$ axioms listed above.

\section{Flag and supersolvable algebras. Examples}\selabel{exemple}
The challenge that remains after providing the theoretical answer
to the ES-problem in \thref{main1} is a purely computational one:
for a given algebra $A$ that is a subspace in a vector space $E$
with a given complement $V$, we have to compute the classifying
object ${\mathcal A} {\mathcal H}^{2}_{A} \, (V, \, A)$ and then
list all algebra structures on $E$ which extend the one of $A$.
For the sake of completeness, we can also compute the space
${\mathcal A} {\mathcal H}^{2} \, (V, \, A)$. In what follows we
provide a way of answering this problem for a large class of such
structures.

\begin{definition} \delabel{flagex}
Let $A$ be an algebra and $E$ a vector space
containing $A$ as a subspace. An algebra
structure on $E$ is called a \emph{flag extending structure} of
$A$ to $E$ if there exists a finite chain of
subalgebras of $E$
\begin{equation} \eqlabel{lant}
E_0 := A \subset E_1 \subset \cdots \subset E_m := E
\end{equation}
such that $E_i$ has codimension $1$ in $E_{i+1}$, for all $i = 0,
\cdots, m-1$. An algebra $E$ that is a flag extending structure of
$k$ will be called a \emph{supersolvable} algebra.
\end{definition}

The notion introduced in the above definition is just the
associative algebra counterpart of the well known concept of
supersolvable Lie algebras (\cite{barnes}). All flag extending
structures of $A$ to $E$ can be completely described by a
recursive reasoning where the key step is $m = 1$. This step
describes and classifies all unified products $A \ltimes V_1$, for
a $1$-dimensional vector space $V_1$. Then, by replacing the
initial algebra $A$ with such a unified product $A \ltimes V_1$,
which will be explicitly described in terms of $A$ only, we can
iterate the process: after $m$ steps, we obtain the description of
all flag extending structures of $A$ to $E$. A special case of
interest for the classification of finite dimensional algebras is
the case when $A = k$, i.e. to classify all $m$-dimensional
supersolvable algebras. In this context we recall that the
classification of solvable Lie algebras, over arbitrary fields,
was achieved only up to dimension four \cite{gra2}. First we need
to introduce the following concept which plays the key role in the
classification of flag extending structures:

\begin{definition} \delabel{tehnica}
Let $A$ be an algebra. A \emph{flag datum} of $A$ is a $6$-tuple
$(\Lambda, \, \lambda, \, D, \, d, \, a_0, \, u)$, where
$\Lambda$, $\lambda : A \to k$ are morphisms of algebras, $D$, $d
: A \to A$ are linear maps, $a_0 \in A$, $u\in k$ satisfying the
following compatibilities:
\begin{eqnarray}
&& \Lambda (a_0) = \lambda (a_0), \quad D(a_0) = d(a_0), \quad
\lambda \circ d = 0, \quad \Lambda \circ D = 0 \eqlabel{flag1} \\
&& d(ab) = a \, d(b) + d(a) \, \lambda (b),  \quad D(ab) = \Lambda
(a) \, D(b) + D(a) \, b \eqlabel{flag2} \\
&& d^2 (a) = u \, d(a) + a \, a_0  - \lambda (a) \, a_0, \quad D^2
(a) = u \, D(a) + a_0 \, a  - \Lambda (a) \, a_0 \eqlabel{flag3} \\
&& D \bigl( d(a) \bigl) - d \bigl( D(a) \bigl) = \bigl( \Lambda
(a) - \lambda (a) \bigl) \, a_0 \eqlabel{flag4} \\
&& \Lambda \bigl( d(a) \bigl) - \lambda \bigl( D(a) \bigl) =
\bigl( \Lambda (a) - \lambda (a) \bigl) \, u \eqlabel{flag45} \\
&& a \, D(b) + \Lambda (b) \, d(a) = d(a) \, b + \lambda (a) \,
D(b) \eqlabel{flag6}
\end{eqnarray}
for all $a$, $b \in A$. We denote by ${\mathcal F} \, (A)
\subseteq {\rm Alg} (A, k)^2 \, \times {\rm Hom}_k (A, A)^2 \,
\times A \times k $ the set of all flag datums of $A$.
\end{definition}

${\mathcal F} \, (A)$ can be the empty set: for instance if the
algebra $A$ has no characters, like in the case of the matrix
algebra $A = M_n(k)$, for $n \geq 2$. The compatibilities
\equref{flag2} show that $d$ and $D$ are twisted derivations of
the algebra $A$. Applying these compatibilities for $a = b = 1_A$
we obtain that $D(1_A) = d(1_A) = 0$, for any $(\Lambda, \,
\lambda, \, D, \, d, \, a_0, \, u) \in {\mathcal F} \, (A)$.

\begin{proposition}\prlabel{unifdim1}
Let $A$ be an algebra and $V$ a vector space of dimension $1$ with
basis $\{x\}$. Then there exists a bijection between the set
${\mathcal A} {\mathcal E} \, (A, V)$ of all algebra extending
structures of $A$ through $V$ and the set ${\mathcal F} \, (A)$ of
all flag datums of $A$.

Through the above bijection, the unified product corresponding to
$(\Lambda, \, \lambda, \, D, \, d, \, a_0, \, u) \in {\mathcal
F}(A)$ will be denoted by $A \ltimes_{(\Lambda, \, \lambda, \, D,
\, d, \, a_0, \, u)} \, x$ and has the multiplication given for
any $a$, $b \in A$ by:
\begin{eqnarray}
(a, 0) \bullet (b, 0) &=& (ab , \, 0), \qquad \,\,\,\,\,\,\,\,\,
(0, x) \bullet (0, x) = (a_0, \, u\, x) \eqlabel{001a}\\
(a, 0) \bullet (0, x) &=& ( d(a), \, \lambda(a) \, x), \,\,\,\,
(0, x) \bullet (a, 0) = (D(a), \, \Lambda(a) x) \eqlabel{002b}
\end{eqnarray}
i.e., $A \ltimes_{(\Lambda, \, \lambda, \, D, \, d, \, a_0, \, u)}
\, x$ is the algebra generated by the algebra $A$ and $x$ subject
to the relations:
\begin{equation}\eqlabel{relflag1}
x^2 = a_0 + u\, x, \quad ax = d(a) + \lambda (a) \, x, \quad xa =
D(a) + \Lambda (a) \, x
\end{equation}
for all $a \in A$.
\end{proposition}

\begin{proof}
We have to compute the set of all bilinear maps $\triangleleft : V
\times A \to V$, $\triangleright : V \times A \to A$,
$\leftharpoonup \, : A \times V \to A$, $ \rightharpoonup \, : A
\times V \to V$, $f: V\times V \to A$ and $\cdot \, : V\times V
\to V $ satisfying the compatibility conditions (A1)-(A12) of
\thref{1}. Since $V$ has dimension $1$ there exists a bijection
between the set of all extending datums of $A$ through $V$ and the
set of all $6$-tuples $(\Lambda, \, \lambda, \, D, \, d, \, a_0,
\, u)$ consisting of four linear maps $\Lambda$, $\lambda: A \to
k$, $D$, $d : A \to A$ and two elements $a_0\in A$ and $u\in k$.
The bijection is given such that the extending datum $\Omega(A, V)
= \bigl(\triangleleft, \, \triangleright, \, \leftharpoonup, \,
\rightharpoonup, \, f, \, \cdot \bigl)$ corresponding to
$(\Lambda, \, \lambda, \, D, \, d, \, a_0, \, u)$ is given by:
\begin{eqnarray*}
x \triangleleft a &:=& \Lambda (a) x, \quad x \triangleright a :=
D(a), \quad a \leftharpoonup x := d(a), \quad a\rightharpoonup x := \lambda (a) x \\
f (x, x) &:=& a_0, \qquad  x \cdot x := u \, x
\end{eqnarray*}
for all $a \in A$. Now, by a straightforward computation one can
see that the axioms (A1)-(A12) of \thref{1} are equivalent to the
fact that $\Lambda$, $\lambda: A \to k$ are algebra maps and the
compatibility conditions \equref{flag1}-\equref{flag6} hold. For
instance, the fact that $(V, \rightharpoonup, \triangleleft)$ is
an $A$-bimodule is equivalent to the fact that $\lambda$ and
$\Lambda : A \to k$ are algebra maps. The axiom (A2) holds if and
only if $\Lambda (a_0) = \lambda (a_0)$, while the axiom (A4) is
equivalent to $\lambda (d(a)) = 0$, for all $a\in A$. The
remaining details are left to the reader.
\end{proof}

\prref{unifdim1} provides an explicit description of all algebras
which contain $A$ as a subalgebra of codimension $1$: they are
isomorphic to an algebra defined by \equref{relflag1}, for some
flag datum $(\Lambda, \, \lambda, \, D, \, d, \, a_0, \, u)$ of
$A$. The existence of this type of algebras depends essentially on
the algebra $A$. Next we will classify this type of algebras by
providing the first explicit classification result of the
ES-problem:

\begin{theorem}\thlabel{clasdim1}
Let $A$ be an algebra of codimension $1$ in the vector space $E$.
Then:

$(1)$ ${\rm Extd} \, (E, A) \cong {\mathcal A} {\mathcal
H}^{2}_{A} \, (k, A) \cong {\mathcal F} \, (A) /\equiv $, where
$\equiv $ is the equivalence relation on the set ${\mathcal F} \,
(A)$ defined as follows: $(\Lambda, \, \lambda, \, D, \, d, \,
a_0, \, u)  \equiv (\Lambda', \, \lambda', \, D', \, d', \, a'_0,
\, u') $ if and only if $\Lambda = \Lambda'$, $\lambda = \lambda'$
and there exists a pair $(q, \alpha) \in k^* \times A$ such that:
\begin{eqnarray}
D(a) &=& q\, D'(a) + \alpha \, a - \Lambda (a) \, \alpha
\eqlabel{eci1} \\
d(a) &=& q \, d'(a) + a \, \alpha - \lambda(a)\, \alpha \eqlabel{eci2} \\
a_0 &=& q^2 \, a'_0 + \alpha^2 - u \, \alpha + q \, d' (\alpha) +
q
D' (\alpha) \eqlabel{eci3} \\
u &=& q\, u' + \lambda'(\alpha) + \Lambda'(\alpha) \eqlabel{eci4}
\end{eqnarray}
for all $a \in A$. The bijection between ${\mathcal F} \,
A)/\equiv$ and ${\rm Extd} \, (E, A)$ is given by:
$$
\overline{(\Lambda, \, \lambda, \, D, \, d, \, a_0, \, u)} \mapsto
A \ltimes_{(\Lambda, \, \lambda, \, D, \, d, \, a_0, \, u)} \, x
$$
where $\overline{(\Lambda, \, \lambda, \, D, \, d, \, a_0, \, u)}$
is the equivalence class of $(\Lambda, \, \lambda, \, D, \, d, \,
a_0, \, u)$ via the relation $\equiv$ and $A \ltimes_{(\Lambda, \,
\lambda, \, D, \, d, \, a_0, \, u)} \, x$ is the algebra defined
by \equref{relflag1}.

$(2)$ ${\rm Extd'} \, (E, A) \cong {\mathcal A} {\mathcal H }^{2}
(k, A) \cong {\mathcal F} \, (A) / \approx $, where $\approx$ is
the following relation on the set ${\mathcal F} \, (A)$:
$(\Lambda, \, \lambda, \, D, \, d, \, a_0, \, u)  \approx
(\Lambda', \, \lambda', \, D', \, d', \, a'_0, \, u') $ if and
only if $\Lambda = \Lambda'$, $\lambda = \lambda'$ and there
exists $\alpha \in A$ such that \equref{eci1}-\equref{eci4} are
fulfilled for $q = 1$. The bijection between
${\mathcal F} \, (A)/\approx$ and ${\rm Extd'} \, (E, A)$
is given by:
$$
\overline{\overline{(\Lambda, \, \lambda, \, D, \, d, \, a_0, \,
u)}} \mapsto A \ltimes_{(\Lambda, \, \lambda, \, D, \, d, \, a_0,
\, u)} \, x
$$
where $\overline{\overline{(\Lambda, \, \lambda, \, D, \, d, \,
a_0, \, u)}}$ is the equivalence class of $(\Lambda, \, \lambda,
\, D, \, d, \, a_0, \, u)$ via $\approx$.
\end{theorem}

\begin{proof}
Let $(\Lambda, \, \lambda, \, D, \, d, \, a_0, \, u)$, $(\Lambda',
\, \lambda', \, D', \, d', \, a'_0, \, u') \in {\mathcal F} \,
(A)$ and $\Omega(A, V)$, respectively $\Omega' (A, V)$ the
corresponding algebra extending structures constructed in the
proof of \prref{unifdim1}. The proof relies on \prref{unifdim1}
and \thref{main1}. Since ${\rm dim}_k (V) = 1$, any linear map $r:
V \to A$ is uniquely determined by an element $\alpha \in A$ such
that $r(x) = \alpha$, where $\{x\}$ is a basis in $V$. On the
other hand, any automorphism $v$ of $V$ is uniquely determined by
a non-zero scalar $q\in k^*$ such such $v (x) = q x$. Based on
these facts, a little computation shows that the compatibility
conditions from \deref{echiaa} and respectively \deref{echiaabb}
take precisely the form given in $(1)$ and $(2)$ above and hence
the proof is finished.
\end{proof}

Using the first statement of \thref{clasdim1} we are now able to
describe the Galois group ${\rm Gal} \, (B/A)$ of an algebra
extension $A \subseteq B$ for which $A$ has codimension $1$ in
$B$.

\begin{corollary} \colabel{galcodim1}
Let $A$ be an algebra and $(\Lambda, \, \lambda, \, D, \, d, \,
a_0, \, u) \in {\mathcal F} \, (A)$ a flag datum of $A$. Then
there exists an isomorphism of groups
$$
{\rm Gal} \, \bigl( A \ltimes_{(\Lambda, \, \lambda, \, D, \, d,
\, a_0, \, u)} \, x / A\bigl) \, \cong \, {\mathbb G}_A \, \bigl(
\Lambda, \, \lambda, \, D, \, d, \, a_0, \, u \bigl)
$$
where ${\mathbb G}_A \, \bigl( \Lambda, \, \lambda, \, D, \, d, \,
a_0, \, u \bigl)$ is the set of all pairs $(\alpha, \, q) \in A
\times k^* $ such that for any $a \in A$:
\begin{eqnarray*}
&& (1 - q) \, D(a) = \alpha \, a - \Lambda (a) \, \alpha, \qquad
(1 - q) \, d(a) = a \, \alpha - \lambda(a)\, \alpha \eqlabel{eci2aa} \\
&& (1-q^2) \, a_0 = \alpha^2 - u \, \alpha + q \, d (\alpha) + q D
(\alpha), \quad (1-q) \, u = \lambda(\alpha) + \Lambda(\alpha)
\eqlabel{eci4aa}
\end{eqnarray*}
with the multiplication given by $(\alpha, \, q) \cdot (\alpha',
\, q') : = (\alpha' + q' \alpha, \, qq')$, for all $(\alpha, \,
q)$, $(\alpha', \, q') \in {\mathbb G}_A \, \bigl( \Lambda, \,
\lambda, \, D, \, d, \, a_0, \, u \bigl)$.
\end{corollary}

\begin{proof}
It follows from \thref{clasdim1} taking into account
\prref{unifdim1} and \coref{galgrup}. The multiplication on
${\mathbb G}_A \, \bigl( \Lambda, \, \lambda, \, D, \, d, \, a_0,
\, u \bigl)$ is the one given by \equref{grupstrc} which in our
context comes down to the desired formula.
\end{proof}

Next we will highlight the efficiency of \thref{clasdim1} in
classifying supersolvable algebras. We start with $A = k$: by
computing ${\mathcal A} {\mathcal H}^{2}_{k} \, (k, k)$ we will
classify in fact all $2$-dimensional algebras over an arbitrary
field $k$ since any algebra map between two $2$-dimensional
algebras automatically stabilizes $k$. Thus the next corollary
originates in \cite{peirce1, study}, where all $2$-dimensional
algebras over the field of complex numbers $\CC$ were classified.
By replacing $\CC$ with an arbitrary field $k$ the situation
changes: the number of isomorphism types of $2$-dimensional
algebras depends heavily on the characteristic of $k$ as well as
on the set $k \setminus k^2$, where $k^2 = \{ q^2 \, | \, q \in k
\}$. First we set the notations which will play the key role in
the classification of flag algebras:

If $k^2 \neq k$, we shall fix $S \subseteq k\setminus k^2$ a
system of representatives for the following relation on $k
\setminus k^2$: $d \equiv d'$ if and only if there exists $q \in
k^{*}$ such that $d = q^{2} d'$. Hence, $|S| = [k^* \, : \,
(k^2)^* ] - 1$, where $[k^* \, : \, (k^2)^* ]$ is the index of
$(k^2)^*$ in the multiplicative group $(k^*, \cdot)$.

If ${\rm char} (k) = 2$ and $k^2 \neq k$ we denote by $R \subseteq
k \setminus k^2$ a system of representatives for the following new
equivalence relation on $k \setminus k^2$: $\delta \equiv_1
\delta'$ if and only if there exists $q \in k^{*}$ such that
$\delta - q^{2} \delta' \in k^2$. Then, $|R| \leq [k^* \, : \,
(k^2)^* ] - 1$.

If ${\rm char} (k) = 2$ we also consider the following equivalence
relation on $k$: $c \equiv_2 c'$ if and only if there exists
$\alpha \in k$ such that $c - c' = \alpha^2 - \alpha$. We shall
fix $T \subseteq k$ a system of representatives for this relation
such that $0\in T$.

Based on \thref{clasdim1} we can prove the following results that
classifies all $2$-dimensional algebras over an arbitrary field.
Of course, $(2)$ and $(3)$ below are well known results. For $(4)$
and $(5)$ we are not able to indicate a reference in the
literature.

\begin{corollary}\colabel{dim2alg}
Let $k$ be an arbitrary field. Then:

$(1)$ There exists a bijection
${\mathcal A} {\mathcal H}^{2}_{k} \, (k, k) \cong k \times k
/\equiv $, where $\equiv $ is the equivalence relation on $k\times
k$ defined as follows: $(a, \, b)  \equiv (a', \, b') $ if and
only if there exists a pair $(q, \alpha) \in k^* \times k$ such
that:
\begin{equation}\eqlabel{reldim2}
a = q^2 \, a' + \alpha^2 - b \, \alpha, \qquad
b = q\, b' + 2 \, \alpha
\end{equation}
The bijection between $k \times k /\equiv $ and the isomorphism
classes of all $2$-dimensional algebras is given by $\overline{(a,
b)} \mapsto  k_{(a, b)}$, where $k_{(a, b)}$ is the algebra having
$\{1, x\}$ as a basis and the multiplication given by $x^2 = a + b
x$.

$(2)$ If ${\rm char} (k) \neq 2$ and $k = k^2$, then the factor
set $k \times k /\equiv $ is equal to $\{ \overline{(0, 0)}, \,
\overline{(0, 1)} \}$. Thus, there exist only two types of
$2$-dimensional algebras, namely $k_{(0, 0)}$ and $k_{(0, 1)}
\cong k \times k$.

$(3)$ If ${\rm char} (k) \neq 2$ and $k \neq k^2$ then the factor
set $k \times k /\equiv $ is equal to $\{\overline{(0, 0)}, \,
\overline{(0, 1)}\} \cup \{\overline{(d, 0)} ~|~ d \in S \} $.
Thus, in this case there exist $1 + [k^* \, : \,  (k^2)^* ]$ types
of isomorphisms of $2$-dimensional algebras namely $k_{(0, 0)}$,
$k_{(0, 1)}$ and $k_{(d, 0)}$, for some $d \in S$.

$(4)$ If ${\rm char} (k) = 2$ and $k = k^2$ then the factor set $k
\times k /\equiv $ is equal to $\{\overline{(0, 0)}\} \, \cup \,
\{\overline{(c, 1)} ~|~ c \in T \} $. Thus, in this case there
exist $1 + |T|$ types of isomorphisms of $2$-dimensional algebras
namely $k_{(0, 0)}$ and $k_{(c, 1)}$, for some $c \in T$.

$(5)$ If ${\rm char} (k) = 2$ and $k \neq k^2$ then the factor set
$k \times k /\equiv $ is equal to $\{\overline{(0, 0)}\} \, \cup
\, \{\overline{(c, 1)} ~|~ c \in T \} \, \cup \,
\{\overline{(\delta, 0)} ~|~ \delta \in R \} $. Thus, in this case
there exist $1 + |T| + |R| $ types of isomorphisms of
$2$-dimensional algebras namely $k_{(0, 0)}$, $k_{(c, 1)}$,
$k_{(\delta, 0)}$, for some $c \in T$ and $\delta \in R$.
\end{corollary}

The algebra $k_{(d, 0)}$, for some $d \in S$ is denoted by
$k(\sqrt{d})$ and is a quadratic field extension of $k$.

\begin{proof}
$(1)$ Since there exists only one algebra map $k \to k$, namely
the identity map, and any linear map $D : k \to k$ with $D(1) = 0$
is the trivial map it follows that ${\mathcal F} \, (k) \cong k
\times k$. Through this identification it is straightforward to
see that the equivalence relation in \thref{clasdim1} takes the
form given by \equref{reldim2}.

$(2)$ The fact that $k \times k /\equiv$ has only two elements,
namely $\{ \overline{(0, 0)}, \, \overline{(0, 1)} \}$ follows
trivially. Let $(a, \,b) \in k \times k$. If $a + 4^{-1} b^2 = 0$,
then we can we denote $b = 2 \alpha$, with $\alpha \in k$. Then it
follows that $(a, \,b) = (- \alpha^{2}, \, 2 \alpha)$, for some
$\alpha \in k$ and thus $(a, \, b) \equiv (0, \, 0)$. On the other
hand, if $a + 4^{-1} b^2 \neq 0$, then we can pick $T \in k^{*}$
such that $a + 4^{-1} b^2 = T^2$, since $k = k^2$. Again, we
denote $b = 2 \alpha$, with $\alpha \in k$ and we obtain $(a, \,b)
= ( T^{2} - \alpha^{2}, \, 2 \alpha)$, and hence $(a, \, b) \equiv
(0, \, 1)$.

$(3)$ If  $k \neq k^2$ we will prove that $k \times k /\equiv $
coincides with $ \{\overline{(0, 0)}, \, \overline{(0, 1)}\} \cup
\{\overline{(d, 0)} ~|~ d \in S \} $. Indeed, consider $(a, \,b)
\in k \times k$. Besides from the two possibilities already
studied in $(2)$ we can also have $a + 4^{-1} b^2 = d$, with $d
\in k \setminus k^2$. As before, we denote $b = 2 \alpha$, with
$\alpha \in k$. It follows that $(a,\, b) = (d - \alpha^{2},\, 2
\alpha)$ and therefore $(a, \, b) \equiv (d, \, 0)$.

$(4)$ and $(5)$ The proof is based on the following observations.
$(a, b) \equiv (0,0)$ if and only if $b = 0$ and $a \in k^2$.
Thus, if $k = k^2$, then we have that $(a, 0) \equiv (0,0)$, for
any $a \in k$. If $k \neq k^2$, then $(a, 0)$ is either equivalent
to $(0,0)$ if $a \in k^2$ or to $(\delta, 0)$, for some $\delta
\in R$ in the case that $a \in k \setminus k^2$. We take into
account that $(\delta, 0) \equiv (\delta', 0)$ if and only if
$\delta \equiv_1 \delta'$.

Let now $(a, b) \in k$, with $b \neq 0$. Then $(a, b) \equiv (a
b^{-2}, 1)$. If $a =0$ the latter is equivalent to $(0, 1)$ and,
if $a \neq 0$, $(a b^{-2}, 1)$ is equivalent to $(c, 1)$, for some
$c \in T$ since $(c, 1) \equiv (c', 1)$ if and only if $c \equiv_2
c'$. This finishes the proof.
\end{proof}

\thref{clasdim1} provides the necessary tool for describing and
classifying supersolvable algebras in a purely computational and
algorithmic way. Having described all $2$-dimensional algebras
over arbitrary fields in \coref{dim2alg} we can now take a step
further and describe all supersolvable algebras of dimension $3$,
i.e. all algebras $E$ for which there exists a chain of
subalgebras
\begin{equation} \eqlabel{lant}
k = E_0 \subset E_1 \subset E_2 = E
\end{equation}
such that ${\rm dim} (E_1) = 2$. The algebras described in this
way will be classified up to an isomorphism that stabilizes $E_1$
by using \thref{clasdim1}. First we should notice that since $E_1$
has dimension $2$, it should coincide with one the following
algebras: $k_{(0, 0)}$, $k_{(0, 1)}$, $k_{(d, 0)} = k(\sqrt{d})$,
or $k_{(c, 1)}$. Now, the algebra $k(\sqrt{d})$ has no characters
(i.e. there exist no algebra maps $k(\sqrt{d}) \to k$). This means
that ${\mathcal F} \, (k(\sqrt{d}))$ is empty, i.e. there exist no
$3$-dimensional algebras containing $k(\sqrt{d})$ as a subalgebra.
Thus, any $3$-dimensional supersolvable algebra $E$ has the middle
term $E_1$ isomorphic to $k_{(0, 0)}$, $k_{(0, 1)}$ if ${\rm char}
(k) \neq 2$ or to $k_{(0, 0)}$, $k_{(c, 1)}$, for some $c\in T$ in
the case when ${\rm char} (k) = 2$. In what follows we describe
all $3$-dimensional supersolvable algebras $E$ that contain and
stabilize $k_{(0, 0)}$ as a subalgebra. First, we need to describe
the set of all flag datums of the algebra $k_{(0,0)}$:

\begin{lemma}\lelabel{flag00}
Let $k$ be a field. Then ${\mathcal F} \, (k_{(0,0)})$ is the
coproduct of the following sets:
$$
{\mathcal F} \, (k_{(0,0)}) := {\mathcal F}_1 \, (k_{(0,0)})
\sqcup {\mathcal F}_2 \, (k_{(0,0)}) \sqcup {\mathcal F}_3 \,
(k_{(0,0)}), \qquad {\rm where:}
$$
${\mathcal F}_1 \, (k_{(0,0)}) \cong k \times k^* \times k$, with
the bijection given such that the flag datum $(\Lambda, \,
\lambda, \, D, \, d, \, a_0, \, u)$ corresponding to $(D_1, \,
a_{01}, \, u) \in k \times k^* \times k$ is defined by:
$$
\lambda (x) = \Lambda (x) := 0, \quad D(x) = d(x) := D_1 x, \quad
a_0 := D^2_1 - u\, D_1 + a_{01}\, x, \quad u:= u
$$
${\mathcal F}_2 \, (k_{(0,0)}) \cong k^2$, with the bijection
given such that the flag datum $(\Lambda, \, \lambda, \, D, \, d,
\, a_0, \, u)$ corresponding to $(D_1, \, u) \in k^2$ is defined
by:
$$
\lambda (x) = \Lambda (x) := 0, \quad D(x) = d(x) := D_1 x, \quad
a_0 := D^2_1 - u\, D_1,  \quad u:= u
$$
${\mathcal F}_3 \, (k_{(0,0)}) \cong \{ (D_1, \, d_1) \in k \times
k \,\,  | \,\, D_1 \neq d_1 \}$ with the bijection given such that
the flag datum  $(\Lambda, \, \lambda, \, D, \, d, \, a_0, \, u)$
corresponding to $(D_1, \, d_1) \in k \times k$ (with $D_1 \neq
d_1$) is defined by:
$$
\lambda (x) = \Lambda (x) := 0, \quad D(x) := D_1 x, \quad d(x) :=
d_1 x, \quad a_0 := - D_1 d_1, \quad u := D_1 + d_1
$$
\end{lemma}

\begin{proof}
The algebra $k_{(0,0)}$ has only one character, namely the map
$\Lambda : k_{(0, 0)} \to k$, defined by $\Lambda (1) = 1$ and
$\Lambda (x) = 0$. A straightforward computation shows that the
set ${\mathcal F} \, (k_{(0,0)})$ of all flag datums of $k_{(0,
0)}$ identifies with the set of all quadruples $(D_1, \, d_1, \,
a_{01}, \, u) \in k^4$ satisfying the following two
compatibilities:
$$
a_{01} \, D_1 = a_{01} \, d_1, \qquad D^2_1 - u \, D_1 =  d^2_1 -
u \, d_1
$$
Under this bijection the flag datum $(\Lambda, \, \lambda, \, D,
\, d, \, a_0, \, u)$ associated to the quadruple $(D_1, \, d_1, \,
a_{01}, \, u) \in k^4$ is given by:
$$
\lambda (x) = \Lambda (x) := 1, \quad D(x) := D_1 x, \quad d(x) :=
d_1 x, \quad a_0 := D^2_1 - u\, D_1 + a_{01}\, x
$$
A detailed discussion on these coefficients (if $a_{01} = 0$ or
$a_{01} \neq 0$) allows us to write ${\mathcal F} \, (k_{(0,0)})$
as the disjoint union of the sets mentioned in the statement.
\end{proof}

The  next result classifies all $3$-dimensional supersolvable
algebras that contain and stabilize $k_{(0,0)}$ as a subalgebra.
The classification does not depend on the characteristic of the
field $k$.

\begin{corollary} \colabel{flag3a}
Let $k$ be an arbitrary field.

$(1)$ If $k = k^2$, then there exist five isomorphism classes of
$3$-dimensional supersolvable algebras that contain and stabilize
$k_{(0,0)}$ as a subalgebra, each of them having the $k$-basis
$\{1, x, y\}$ and relations given as follows:
\begin{eqnarray*}
&A^{0}_1:& \qquad x^2 = 0, \quad y^2 = y, \qquad \quad xy = x, \quad yx = 0; \\
&A^{0}_2:& \qquad x^2 = 0, \quad y^2 = y, \qquad \quad xy = yx = 0; \\
&A^{0}_3:& \qquad x^2 = 0, \quad y^2 = 0, \qquad \quad xy = yx = 0; \\
&A^{0}_4:& \qquad x^2 = 0, \quad y^2 = x, \qquad \quad xy = yx = 0; \\
&A^{0}_5:& \qquad x^2 = 0, \quad y^2 = x + y, \quad \,\, xy = yx =
0
\end{eqnarray*}
$(2)$ If $k \neq k^2$, then there exist $4 + [k^* \, : \, (k^2)^*
] $ isomorphism classes of $3$-dimensional supersolvable algebras
that contain and stabilize $k_{(0,0)}$ as a subalgebra, namely
those from $(1)$ and an additional family defined for any $d \in
S$ by the relations:\footnote{We recall that $|S| = [k^* \, : \,
(k^2)^* ] - 1$}
$$
A^{0} (d): \qquad \quad x^2 = 0, \quad y^2 = d\, x, \quad xy = yx
= 0
$$
\end{corollary}

\begin{proof}
The proof is similar to the one of \coref{dim2alg}. Indeed, in
\leref{flag00} we have described the set of all flag datums
${\mathcal F} \, (k_{(0,0)})$. The equivalence relation from
\thref{clasdim1} takes the following form for each of the sets
${\mathcal F}_i \, (k_{(0,0)})$, $i = 1, 2, 3$:

$\bullet$ For ${\mathcal F}_3 \, (k_{(0,0)})$: $(D_1, d_1) \equiv
(D'_1, d'_1)$ if and only if there exists $(q, \alpha) \in k^*
\times k$ such that
$$
D_1 = q D'_1 + \alpha, \qquad d_1 = q d'_1 + \alpha
$$
In this case the factor set ${\mathcal F}_3 \, (k_{(0,0)})/
\equiv$ is a singleton having $\overline{(0, 1)}$ as the only
element. The algebra associated to $\overline{(0, 1)}$, i.e. the
unified product $k_{(0,0)} \ltimes_{(\Lambda, \, \lambda, \, D, \,
d, \, a_0, \, u)} \, y$ from \thref{clasdim1}, is precisely the
noncommutative algebra $A^{0}_1$.

$\bullet$ For ${\mathcal F}_2 \, (k_{(0,0)})$: $(D_1, u) \equiv
(D'_1, u')$ if and only if there exists $(q, \alpha) \in k^*
\times k$ such that
$$
D_1 = q D'_1 + \alpha, \qquad u = q u' + 2 \alpha
$$
We can easily show that, regardless of the characteristic of $k$,
the factor set ${\mathcal F}_2 \, (k_{(0,0)})/ \equiv $ has two
elements namely $\{\overline{(0, 0)}, \, \overline{(0, 1)} \}$.
The associated unified products \\ $k_{(0,0)} \ltimes_{(\Lambda,
\, \lambda, \, D, \, d, \, a_0, \, u)} \, y$ are the algebras
$A^{0}_2$ and $A^{0}_3$.

$\bullet$ For ${\mathcal F}_1 \, (k_{(0,0)})$: $(D_1, a_{01}, u)
\equiv (D'_1, a'_{01}, u') \in k\times k^* \times k $ if and only
if there exists $(q, \alpha_0, \alpha_1) \in k^* \times k \times
k$ such that
\begin{equation} \eqlabel{altrilecaz}
D_1 = q D'_1 + \alpha_0, \quad u = q u' + 2 \alpha_0, \quad a_{01}
= q^2 \, a'_{01} - q \alpha_1 u' + 2 q \alpha_1 D'_1
\end{equation}
Suppose first that ${\rm char} (k) \neq 2$. In order to compute
the factor set ${\mathcal F}_1 \, (k_{(0,0)})/ \equiv$ we
distinguish two cases. First, if $k = k^2$, we can easily show
that ${\mathcal F}_1 \, (k_{(0,0)})/ \equiv $ has two elements
namely $\{\overline{(0, 1, 0)}, \overline{(0, 1, 1)} \}$; the
unified products $k_{(0,0)} \ltimes_{(\Lambda, \, \lambda, \, D,
\, d, \, a_0, \, u)} \, y$ associated to $\overline{(0, 1, 0)}$
and $\overline{(0, 1, 1)}$ are precisely the algebras $A^{0}_4$
and $A^{0}_5$.

Secondly, if $k \neq k^2$, it is straightforward to see that $
{\mathcal F}_1 \, (k_{(0,0)})/ \equiv $ is precisely the set
$\{\overline{(0, 1, 0)}, \overline{(0, 1, 1)} \} \, \cup \,  \{
\overline{(0, d, 0)} \, | \, d \in S \}$. Moreover, the unified
product associated to $\overline{(0, d, 0)}$, for some $d\in S$,
is the algebra $A^{0}(d)$.

Assume now that ${\rm char} (k) = 2$. Then the equivalence
relation given by \equref{altrilecaz} takes the form:
$$
D_1 = q D'_1 + \alpha_0, \quad u = q u', \quad a_{01} = q^2 \,
a'_{01} - q \alpha_1 u'
$$
The factor set ${\mathcal F}_1 \, (k_{(0,0)})/ \equiv$ is the same
as in the case ${\rm char} (k) \neq 2$. This can be easily seen
from the following observations: for any $u \neq 0$, we have that
$(D_1, a_{01}, u) \equiv (0, 1, 1)$ and $(D_1, a_{01}, 0) \equiv
(0, 1, 0)$ if and only if $a_{01} \in (k^*)^2$. Finally, for $d
\in k \setminus k^2$ we have that $(D_1, d, 0) \equiv (0, d, 0)$
and the proof is finished since $(0, d, 0) \equiv (0, d', 0)$ if
and only if there exists $q \in k^*$ such that $d = q^2 d'$.
\end{proof}

Next we will classify all supersolvable algebras of dimension $3$
that contain and stabilize $k_{(0,1)}$ as a subalgebra: remark
that even in characteristic $2$ this algebra still exists as an
intermediary algebra corresponding to $k_{(c,1)}$, for $c := 0 \in
T$, since we have chosen a system of representatives $T$ which
contains $0$. First we need the following:

\begin{lemma}\lelabel{flag01}
Let $k$ be a field. Then ${\mathcal F} \, (k_{(0,1)})$ is the
coproduct of the following sets:
$$
{\mathcal F} \, (k_{(0,1)}) := {\mathcal F}_1 \, (k_{(0,1)})
\sqcup {\mathcal F}_2 \, (k_{(0,1)}) \sqcup {\mathcal F}_3 \,
(k_{(0,1)})\sqcup {\mathcal F}_4 \, (k_{(0,1)}), \qquad {\rm
where:}
$$
${\mathcal F}_1 \, (k_{(0,1)}) \cong k^3$ with the bijection given
such that the flag datum $(\Lambda, \, \lambda, \, D, \, d, \,
a_0, \, u)$ corresponding to $(D_1, a_{01}, u) \in k^3$ is defined
by
$$
\Lambda (x) = \lambda (x) := 0, \quad D(x) = d(x) := D_1 \, x,
\quad a_0 := D^2_1 - u\, D_1 - a_{01} + a_{01} \, x, \quad u:=u
$$
${\mathcal F}_2 \, (k_{(0,1)}) \cong k^3$ with the bijection given
such that the flag datum $(\Lambda, \, \lambda, \, D, \, d, \,
a_0, \, u)$ corresponding to $(D_1, a_{01}, u) \in k^3$ is defined
by
$$
\Lambda (x) = \lambda (x) := 1, \,\, D(x) = d(x) := D_1 (1 - x),
\,\, a_0 := D_1^2 + u D_1 + a_{01} \,x, \,\,\, u := u
$$
${\mathcal F}_3 \, (k_{(0,1)}) \cong k^2$ with the bijection given
such that the flag datum $(\Lambda, \, \lambda, \, D, \, d, \,
a_0, \, u)$ corresponding to $(D_1, d_1) \in k^2$ is defined by
$$
\Lambda (x) := 0, \,\,  \lambda (x) := 1, \,\, D(x) := D_1 \, x,
\,\, d(x) := d_1 (1 -x), \,\, a_0 := D_1\, d_1, \,\, u := D_1 -
d_1
$$
${\mathcal F}_4 \, (k_{(0,1)}) \cong k^2$ with the bijection given
such that the flag datum $(\Lambda, \, \lambda, \, D, \, d, \,
a_0, \, u)$ corresponding to $(D_1, d_1) \in k^2$ is defined by
$$
\Lambda (x) := 1, \,\,  \lambda (x) := 0, \,\, D(x) := D_1 (1 -
x), \,\, d(x) := d_1 \, x, \,\, a_0 := D_1\, d_1, \,\, u := d_1 -
D_1
$$
\end{lemma}

\begin{proof}
Since $x^2 = x$ it follows that the algebra $k_{(0,1)}$ has only
two characters, namely the maps that send $x \mapsto 0$ and
respectively $x \mapsto 1$. Thus, in order to compute ${\mathcal
F} \, (k_{(0,1)})$ we distinguish four cases depending on the
characters $(\Lambda, \lambda)$. More precisely, ${\mathcal F}_1
\, (k_{(0,1)})$ will parameterize all flag datums for which
$\Lambda (x) = \lambda (x) = 0$ while ${\mathcal F}_2 \,
(k_{(0,1)})$ those for which $\Lambda (x) = \lambda (x) = 1$. We
are left with two more cases: ${\mathcal F}_3 \, (k_{(0,1)})$
(resp. ${\mathcal F}_4 \, (k_{(0,1)})$) parameterizes all flag
datums for which $\Lambda (x) := 0$ and $\lambda (x) = 1$ (resp.
$\Lambda (x) = 1$ and $\lambda (x) = 0$). The conclusion follows
in a straightforward manner by applying \deref{tehnica}.
\end{proof}

Now we shall classify all $3$-dimensional supersolvable algebras
that contain and stabilize $k_{(0,1)}$ as a subalgebra. This time,
the classification depends heavily on the base field $k$.

\begin{corollary} \colabel{flag3b}
The isomorphism classes of all $3$-dimensional supersolvable
algebras that contain and stabilize $k_{(0,1)}$ as a subalgebra
are given as follows:

$(1)$ If ${\rm char} (k) \neq 2$ and $k = k^2$, then there are
five such isomorphism classes, namely the algebras with $k$-basis
$\{1, x, y\}$ and relations given as follows:
\begin{eqnarray*}
&A^{1}_1:& \qquad x^2 = x, \quad y^2 = 0, \qquad \quad xy = yx = 0 \\
&A^{1}_2:& \qquad x^2 = x, \quad y^2 = x-1, \quad \,\,  xy = yx = 0 \\
&A^{1}_3:& \qquad x^2 = x, \quad y^2 = 0, \qquad \quad xy = yx = y \\
&A^{1}_4:& \qquad x^2 = x, \quad y^2 = x, \qquad \quad xy = yx = y \\
&A^{1}_5:& \qquad x^2 = x, \quad y^2 = 0, \qquad \,\,\,\,\,\, xy =
y, \quad yx = 0
\end{eqnarray*}
$(2)$ If ${\rm char} (k) \neq 2$ and $k \neq k^2$, then there
exist $3 + 2\, [k^* \, : \,  (k^2)^* ]$ such isomorphism classes.
These are the five types from (1) and two additional families
defined by the following relations for any $d\in S$:
\begin{eqnarray*}
&B_{1}(d):& \qquad x^2 = x, \quad y^2 = d \,(x-1), \qquad xy = yx = 0 \\
&B_{2}(d):& \qquad x^2 = x, \quad y^2 = d \, x, \qquad \qquad
\,\,\, xy = yx = y
\end{eqnarray*}

$(3)$ If ${\rm char} (k) = 2$ and $k = k^2$, then there exist $3 +
2\, |T|$ such isomorphism classes, namely the algebras having
$\{1, x, y\}$ as a $k$-basis and subject to the following
relations for any $c\in T$:
\begin{eqnarray*}
C^{1}_1&:& \qquad x^2 = x, \quad y^2 = 0, \qquad \qquad \qquad \,\, xy = yx = 0 \\
C^{1}_2(c)&:& \qquad x^2 = x, \quad y^2 = c + (c+1)\, x, \quad \,\,\, xy = yx = 0 \\
C^{1}_3&:& \qquad x^2 = x, \quad y^2 = 0, \qquad \quad \qquad \quad \, xy = yx = y \\
C^{1}_4(c)&:& \qquad x^2 = x, \quad y^2 = y + c\, x, \qquad \qquad   xy = yx = y \\
C^{1}_5&:& \qquad x^2 = x, \quad y^2 = 0, \qquad \qquad \qquad
\,\, \, xy = y, \quad yx = 0
\end{eqnarray*}

$(4)$ If ${\rm char} (k) = 2$ and $k \neq k^2$, then there exist
$3 + 2\, |T| + 2 |R|$ such isomorphism classes. These are the
types from (3) and two additional families defined by the
following relations for any $\delta \in R$:
\begin{eqnarray*}
&D_{1}(\delta):& \qquad x^2 = x, \quad y^2 = \delta \,(x+1), \qquad xy = yx = 0 \\
&D_{2}(\delta):& \qquad x^2 = x, \quad y^2 = \delta \, x, \qquad
\qquad \,\,\, xy = yx = y
\end{eqnarray*}
\end{corollary}

\begin{proof}
We shall use the description of ${\mathcal F} \, (k_{(0, 1)})$
given in \leref{flag01}. First of all we remark that, due to
symmetry, the unified products associated to the flag datums of
${\mathcal F}_4 \, (k_{(0,1)})$ are isomorphic to the ones
associated to ${\mathcal F}_3 \, (k_{(0,1)})$. Thus we only have
to analyze the cases ${\mathcal F}_i \, (k_{(0,1)})$, with $i = 1,
2, 3$. The equivalence relation from \thref{clasdim1}, applied to
each of the sets $k^3$ and $k^2$, takes the following form:

$\bullet$ For ${\mathcal F}_1 \, (k_{(0,1)}) \cong k^3$: $(D_1,
a_{01}, u) \equiv (D'_1, a'_{01}, u')$ if and only if there exists
$(q, \alpha_0, \alpha_1) \in k^* \times k \times k$ such that
\begin{eqnarray}
D_1 &=& q \, D'_1 + \alpha_0 + \alpha_1 \eqlabel{1000}\\
a_{01} &=& q^2 \, a'_{01} + \alpha^2_1 - q \, u' \alpha_1 + 2 q \,
\alpha_1 D'_1 \eqlabel{1001} \\
u &=& q\, u' + 2 \, \alpha_0 \eqlabel{1002}
\end{eqnarray}
Suppose first that ${\rm char} (k) \neq 2$. Then we have: $(D_1,
a_{01}, u) \equiv (0, 0, 0)$ if and only if $a_{01} = (D_1 -
2^{-1} \, u)^2$ and $(D_1, a_{01}, u) \equiv (0, 1, 0)$ if and
only if $a_{01} - (D_1 - 2^{-1} \, u)^2 \in (k^*)^2$. These two
observations show that if $k = k^2$ then the factor set $k^3/
\equiv $ has two elements, namely $\{\overline{(0, 0, 0)},
\overline{(0, 1, 0)} \}$. The unified products $k_{(0,1)}
\ltimes_{(\Lambda, \, \lambda, \, D, \, d, \, a_0, \, u)} \, y$
associated to $\overline{(0, 0, 0)}$ and respectively
$\overline{(0, 1, 0)}$ are the algebras $A^1_1$ and $A^1_2$. On
the other hand, if $k \neq k^2$, then the factor set $k^3/ \equiv
$ is equal to $\{\overline{(0, 0, 0)}, \overline{(0, 1, 0)} \} \,
\cup \, \{ (0, d, 0) \, | \, d \in S \}$. This can be easily seen
from the following observations: $(D_1, a_{01}, u) \equiv (0, d,
0)$ if and only if $a_{01} - (D_1 - 2^{-1} \, u)^2 = q^2 d$, for
some $q\in k^*$ while  $(0, d, 0) \equiv (0, d', 0)$ if and only
if $d = q^2 d'$, for some $q\in k^*$. The unified product
associated to $\overline{(0, d, 0)}$ is the algebra $B_1(d)$.

Assume now that ${\rm char} (k) = 2$. Then the equivalence
relation on $k^3$ given by \equref{1000}-\equref{1002} takes the
form: $(D_1, a_{01}, u) \equiv (D'_1, a'_{01}, u')$ if and only if
there exists $(q, \alpha_0, \alpha_1) \in k^* \times k \times k$
such that
$$
D_1 = q \, D'_1 + \alpha_0 + \alpha_1, \qquad a_{01} = q^2 \,
a'_{01} + \alpha^2_1 - q \, u' \alpha_1, \qquad u = q\, u'
$$
We will prove the following: if $k = k^2$, then the factor set
$k^3 / \equiv $ is equal to $\{ \overline{(0, 0, 0)}, \\
\overline{(0, c, 1)} \, | \, c \in T \}$ while if $k \neq k^2$,
then the factor set $k^3 / \equiv $ turns out to be $\{
\overline{(0, 0, 0)}, \, \overline{(0, c, 1)} \\ \, | \, c \in T
\} \, \cup \, \{ \overline{(0, \delta, 0)} \, | \, \delta \in R
\}$. Indeed, the above equalities are consequences of the
following observations: $(D_1, a_{01}, u) \equiv (0, 0, 0)$ if and
only if $u = 0$ and $a_{01} \in k^2$; $(D_1, a_{01}, u) \equiv (0,
\delta, 0)$ if and only if $u = 0$ and $a_{01} \equiv_1 d$. Now,
for any $u \neq 0$ we have that $(D_1, a_{01}, u) \equiv (u^{-1}
\, D_1, u^{-2} \, a_{01}, 1) $, and moreover $(D_1, a_{01}, 1)
\equiv (0, c, 1)$, for some $c \in T$. Finally, $(0, c, 1) \equiv
(0, c', 1)$ if and only if $c \equiv_2 c'$. The algebras
associated to $\overline{(0, 0, 0)}$, $\overline{(0, c, 1)}$ and
$\overline{(0, \delta, 0)}$ are $C^1_1$, $C^1_2(c)$ and
respectively $D_1(\delta)$.

$\bullet$ For ${\mathcal F}_2 \, (k_{(0,1)}) \cong k^3$: $(D_1,
a_{01}, u) \equiv (D'_1, a'_{01}, u')$ if and only if there exists
$(q, \alpha_0, \alpha_1) \in k^* \times k \times k$ such that
\begin{eqnarray*}
D_1 &=& q \, D'_1 - \alpha_0 \\
a_{01} &=& q^2 \, a'_{01} - \alpha_1^2 - q\, \alpha_1 u' - 2q
\alpha_1 D'_1 \\
u &=& q\, u' + 2 (\alpha_0 + \alpha_1)
\end{eqnarray*}
Suppose first that ${\rm char} (k) \neq 2$. Then we have: $(D_1,
a_{01}, u) \equiv (0, 0, 0)$ if and only if $a_{01} = - (D_1 +
2^{-1} \, u)^2$ and $(D_1, a_{01}, u) \equiv (0, 1, 0)$ if and
only if $a_{01} = - (D_1 + 2^{-1} \, u)^2 + q^2$, for some $q\in
k^*$. These two observations show that if $k = k^2$ then the
factor set $k \times k \times k / \equiv $ has two elements,
namely $\overline{(0, 0, 0)}$ and $\overline{(0, 1, 0)}$. The
unified products $k_{(0,1)} \ltimes_{(\Lambda, \, \lambda, \, D,
\, d, \, a_0, \, u)} \, y$ associated to $\overline{(0, 0, 0)}$
and $\overline{(0, 1, 0)}$ are the algebras $A^1_3$ and
respectively $A^1_4$. On the other hand, if $k \neq k^2$ then the
factor set $k \times k \times k / \equiv $ is equal to
$\{\overline{(0, 0, 0)}, \overline{(0, 1, 0)} \} \, \cup \, \{ (0,
d, 0) \, | \, d \in S \}$. Indeed, let $(D_1, a_{01}, u) \in k^3$
such that $a_{01} = - (D_1 + 2^{-1} \, u)^2 + d$, for some $d \neq
0$. Then $(D_1, a_{01}, u) \equiv (0, d, 0)$ and moreover $(0, d,
0) \equiv (0, d', 0)$ if and only if there exists $q\in k^*$ such
that $d = q^2 d'$. The unified product $k_{(0,1)}
\ltimes_{(\Lambda, \, \lambda, \, D, \, d, \, a_0, \, u)} \, y$
associated to $\overline{(0, d, 0)}$, for some $d\in S$, is the
algebra $B_2(d)$.

Consider now the case ${\rm char} (k) = 2$. We have: $(D_1,
a_{01}, u) \equiv (0, 0, 0)$ if and only if $u = 0$ and $a_{01}
\in k^2$; $(D_1, a_{01}, u) \equiv (0, \delta, 0)$ if and only if
$u = 0$ and $a_{01} \equiv_1 \delta$. Let $(D_1, a_{01}, u) \in
k^3$ with $u \neq 0$, then $(D_1, a_{01}, u) \equiv (u^{-1}D_1,
u^{-2} a_{01}, 1)$. Finally, $(D_1, a_{01}, 1) \equiv (D'_1,
a'_{01}, 1)$ if and only if $a_{01} \equiv_2 a'_{01}$. These
observations show that if $k = k^2$, then the factor set $k^3 /
\equiv $ is equal to $\{ \overline{(0, 0, 0)}, \, \overline{(0, c,
1)} \, | \, c \in T \}$ while if $k \neq k^2$, the factor set $k^3
/ \equiv $ coincides with $\{ \overline{(0, 0, 0)}, \,
\overline{(0, c, 1)} \, | \, c \in T \} \, \cup \, \{
\overline{(0, \delta, 0)} \, | \, \delta \in R \}$. The algebras
corresponding to $\overline{(0, 0, 0)}$, $\overline{(0, c, 1)}$
and $\overline{(0, \delta, 0)}$ are $C^1_3$, $C^1_4(c)$ and
respectively $D_2(\delta)$.

$\bullet$ For ${\mathcal F}_3 \, (k_{(0,1)})\cong k^2$: $(D_1,
d_1) \equiv (D'_1, d'_1)$ if and only if there exists a triple
$(q, \alpha_0, \alpha_1) \in k^* \times k \times k$ such that
$$
D_1 = q\, D'_1 + \alpha_0 + \alpha_1, \quad d_1 = q\, d'_1 -
\alpha_0
$$
Regardless of the characteristic of $k$, the factor set $k \times
k / \equiv $ contains only one element, namely $\overline{(0,
0)}$. The algebra associated to it is $A^1_5$, if ${\rm char} (k)
\neq 2$ or $C^1_5$, if ${\rm char} (k) = 2$. The proof is now
finished.
\end{proof}

In order to complete the description of all $3$-dimensional
supersolvable algebras we are left to study one more case, namely
the one when ${\rm char} (k) = 2$ and the intermediary algebra is
isomorphic to $k_{(c, 1)}$, for some $c\in T$. The case $c = 0$ is
treated in \coref{flag3b}. Suppose now that $c \neq 0$. We will
see that the set ${\mathcal F} \, (k_{(c, 1)})$ of all flag datums
is empty (i.e. there exists no $3$-dimensional algebras containing
$k_{(c, 1)}$ as a subalgebra) since the algebra $k_{(c, 1)}$ has
no characters. Indeed, the algebra $k_{(c, 1)}$ is generated by an
element $x$ such that $x^2 = c + x$. Therefore, the characters of
this algebra are in bijection with the solutions \emph{in} $k$ of
the equation $y^2 + y + c = 0$. Now, since $c \in T$ and $c \neq
0$ this equation has no solutions in $k$: if $\alpha$ is such a
solution then $c = \alpha + \alpha^2$, i.e. $c \equiv_2 0$ and we
obtain $c = 0$ since $c \in T$ a system o representatives for
$\equiv_2$ containing $0$. As a conclusion of this section we
arrive at the classification of all $3$-dimensional supersolvable
algebras:

\begin{theorem} \thlabel{flag3dimsin}
Let $k$ be an arbitrary field. Then any $3$-dimensional
supersolvable algebra is isomorphic to an algebra from the
following list:

$(1)$ $A^0_i$ or $A^1_i$, for all $i = 1, \cdots, 5$, if ${\rm
char} (k) \neq 2$ and $k = k^2$.

$(2)$ $A^0_i$, $A^1_i$, $A^0(d)$, $B_1(d)$ or $B_2(d)$, for all $i
= 1, \cdots, 5$ and $d\in S$, if ${\rm char} (k) \neq 2$ and $k
\neq k^2$.

$(3)$ $A^0_i$, $C^{1}_1$, $C^{1}_2(c)$, $C^{1}_3$, $C^{1}_4(c)$ or
$C^{1}_5$, for all $i = 1, \cdots, 5$ and $c\in T$, if ${\rm char}
(k) = 2$ and $k = k^2$.

$(4)$ $A^0_i$, $A^0(d)$, $C^{1}_1$, $C^{1}_2(c)$, $C^{1}_3$,
$C^{1}_4(c)$, $C^{1}_5$, $D_1(\delta)$ or $D_2(\delta)$, for all
$i = 1, \cdots, 5$, $d\in S$, $c\in T$ and $\delta \in R$, if
${\rm char} (k) = 2$ and $k \neq k^2$.
\end{theorem}

\end{document}